\documentclass[preprint,12pt]{article}
\usepackage{amsmath,amsthm,amssymb,amscd,fancybox,epsfig,float,subfigure}
\usepackage[margin=1in]{geometry}
\usepackage{authblk}

\usepackage{arydshln}
\usepackage{graphics}
\usepackage{epstopdf}
\usepackage{multirow}
\usepackage{hhline}
\usepackage{cleveref}
\usepackage{xcolor}
\usepackage{tikz}
\usetikzlibrary{arrows,automata}
\usepackage{cite}
\newcommand{\beqn}{\begin{equation}}
\newcommand{\eeqn}{\end{equation}}

\newcommand{\cS}{\mathcal{S}}

\newcommand{\cD}{\mathcal{D}}

\newcommand{\bR}{\mathbb{R}}

\newcommand{\bd}{{\boldsymbol d}}
\newcommand{\bv}{{\boldsymbol v}}

\newcommand{\bp}{{\boldsymbol p}}

\newcommand{\bw}{{\boldsymbol w}}

\newcommand{\bC}{{\bf C}}
\newcommand{\bB}{{\bf B}}

\newcommand{\bL}{\mathbb{L}}

\newcommand{\bx}{\boldsymbol{x}}
\newcommand{\bn}{\boldsymbol{n}}
\newcommand{\by}{\boldsymbol{y}}
\newcommand{\bh}{\boldsymbol{h}}

\newcommand{\eps}{\varepsilon}

\newcommand{\sinb}{\sin{(\pi \beta)}}
\newcommand{\sint}{\sin{(\theta)}}
\newcommand{\cost}{\cos{(\theta)}}

\newcommand{\bbmat}{\begin{bmatrix}}
\newcommand{\ebmat}{\end{bmatrix}}

\newcommand{\spt}[1]{\sin{(\pi #1)}}
\newcommand{\sptsq}[1]{\sin^2{(\pi #1)}}
\newcommand{\sinsq}[1]{\sin^2{(#1)}}

\newcommand{\Adir}{\mathcal{A}_{\text{dir}}}
\newcommand{\Aneu}{\mathcal{A}_{\text{neu}}}
\newcommand{\Kdir}{\mathcal{K}_{\text{dir}}}
\newcommand{\Kneu}{\mathcal{K}_{\text{neu}}}
\newcommand{\snneu}{S^{\text{neu}}_{N}}
\newcommand{\tsnneu}{\tilde{S}^{\text{neu}}_{N}}
\newcommand{\bCdiag}{\boldsymbol C_{\text{diag}}}

\newcommand{\kjer}{\boldsymbol k_{\textrm{AB}}}
\newcommand{\smin}{s_{\text{min}}}
\newcommand{\smax}{s_{\text{max}}}

\newtheorem{definition}{Definition}[section]
\newtheorem{remark}{Remark}[section]
\newtheorem{conjecture}{Conjecture}[section]
\newtheorem{theorem}{Theorem}[section]
\newtheorem{lemma}{Lemma}[section]

\title{On the solution of Laplace's equation in the vicinity of
triple-junctions}
\author{Jeremy Hoskins\thanks{
Applied Mathematics Program, Yale University, USA.\\ 
email: jeremy.hoskins@yale.edu} ,\, 
Manas Rachh\thanks{Center for Computational Mathematics,
Flatiron Institute, USA. \\
email: mrachh@flatironinstitute.org}}

\date{}

\begin{document}

\maketitle
\abstract{
In this paper we characterize the behavior of solutions to systems of boundary 
integral equations associated with Laplace transmission problems 
in composite media consisting of regions with polygonal boundaries. In particular 
 we consider triple junctions, i.e. points at which three distinct
 media meet. We show that, under suitable conditions, solutions to the boundary integral equations in the vicinity 
 of a triple junction are well-approximated by linear combinations of functions of the form $t^\beta,$ where $t$ is the distance of the point from the junction and the powers $\beta$ depend only on the material properties of the media and the angles at which their boundaries meet. Moreover, we use this analysis to design 
 efficient discretizations of boundary integral equations for Laplace transmission problems 
 in regions with triple junctions and demonstrate the accuracy and efficiency of this algorithm 
 with a number of examples.
}

\section{Introduction \label{sec:intro}}

Composite media, i.e.  media consisting of multiple materials in close proximity 
or contact, are both ubiquitous in nature and fascinating in applications since 
their macroscopic properties can be substantially different than those of 
their components. One property of particular interest is the electrostatic 
response of composite media, typically the  
electric potential in the medium which is produced by an externally-applied 
time-independent electric field. 
In such situations one often assumes that the associated electric 
potential satisfies 
Laplace's equation in the interior of each medium and that along each edge
where two media meet one prescribes the jump in the normal derivative of the 
potential. 
Typically the potentials in these jump relations appear multiplied by 
coefficients depending on the electric permittivity. 
This leads to a collection of coupled partial differential equations (PDEs). In addition to classical electrostatics problems, the same equations also arise in, among other things, percolation theory, homogenization theory, and 
the study field enhancements in vacuum insulators
(see, for example, ~\cite{lee2008pseudo,fredkin2003resonant,milton2002theory,tully2007electrostatic,tuncer2002dielectric,fel2000isotropic,ovchinnikov2004conductivity}).

Using classical potential theory this set of partial differential equations (PDEs) can be 
reduced to a system of second-kind boundary integral equations (BIEs). 
In particular, the solution to the PDE in each region is represented as a 
linear combination of a single-layer and a double-layer potential on the 
boundary of each subregion. 
If the edges of the media are smooth then the corresponding kernels in the 
integral equation are as well. Near corners, however, the solutions to both the differential equations
 and the integral equations can 
develop singularities. 

Analytically, the behavior of solutions to both the PDEs and BIEs have
been the subject of extensive analysis
(see, for example
~\cite{craster2004three,keller1987effective,helsing1991transport,chung2005theoretical,schachter1998analytic,berggren2001new,techaumnat2002effect,afanas2004computation,greengard2012stable,claeys2015second}). 
In particular, the existence and uniqueness of solutions in an $\mathbb{L}^2$-sense 
is well-known, under certain natural assumptions on the material 
properties \cite{claeys2017second,mclean2000strongly}. 
Moreover, the asymptotic form of the singularities in the vicinity of a 
junction has been determined for the solutions of both the PDE and its 
corresponding BIE  \cite{chung2005theoretical,schachter1998analytic,craster2004three,milton1981transport,helsing2011effective}. 

Computationally the singular nature of the solutions poses significant 
challenges for many existing numerical methods for solving both the PDEs 
and BIEs. 
Typical approaches involve introducing many additional degrees of freedom 
near the junctions which can impede the speed of the solver and impose 
prohibitive limits on the size and complexity of geometries which can 
be considered. 
{\it Recursive compressed inverse preconditioning (RCIP)} is one way of circumventing the 
difficulty introduced by the presence of junctions in the BIE formulation~\cite{helsing2013solving}. 
In this approach, the extra degrees of freedom introduced by the refinement 
near the junctions are eliminated from the linear system. 
Moreover, the compression and refinement are performed concomitantly 
for multiple junctions in parallel. 
This approach gives an algorithm which scales linearly 
in the number of degrees of freedom added to resolve the singularities 
near the junction. 
The resulting linear system has essentially the same number of 
degrees of freedom as it would if the junctions were absent.

In this paper we restrict our attention to the case of triple junctions, 
extending the existing analysis by showing that under suitable restrictions 
the solution to the BIEs can be well-approximated in the vicinity of a 
triple junction by a linear combination of $t^{\beta_{j}}$, where
$t$ is the distance from the triple junction and the $\beta_{j}$'s are a countable 
collection of real numbers defined implicitly by an equation 
depending only on the angles at which the interfaces meet and the material 
properties of the corresponding media. 
This analysis enables the construction of an efficient computational 
algorithm for solving Laplace's equation in regions with multiple junctions. 
In particular, using this representation we construct an accurate and 
efficient quadrature scheme for the BIE which requires no refinement 
near the junction. 
The properties of this discretization are illustrated with a number 
of numerical examples.

This paper is organized as follows. 
In~\cref{sec:boundary} we
state the boundary value problem for the Laplace triple junction transmission problem, 
summarize relevant properties of layer potentials, and describe the reduction
 of the boundary value problem to a system of boundary integral equations.
In~\cref{sec:mainres} we present the main theoretical results
of this work, the proofs of which 
are given in~\cref{sec:appendix-dir,sec:appendix-neu}. 
In~\cref{sec:conj} we discuss two conjectures extending the results of~\cref{sec:mainres} 
based on extensive numerical evidence.
In~\cref{sec:numdis}, we describe a Nystr\"{o}m discretization which exploits explicit knowledge of the structure
of solutions to the integral equations in the vicinity of triple junctions, and
in~\cref{sec:numres} we demonstrate its effectiveness of numerical solvers. 
Finally, in~\cref{sec:conc} we summarize the results and outline 
directions for future research.


 \section{Boundary value problem \label{sec:boundary}}

Consider a composite medium consisting of a set of $n$ polygonal domains 
$\Omega_1,\dots,\Omega_n$ (see Figure \ref{fig:comp_fig}) with boundaries 
consisting  of $m$ edges $\Gamma_1,\dots,\Gamma_m$ and $k$ vertices 
$\bv_1,\dots,\bv_k$. 
For a given edge $\Gamma_i$ let $L_i$ denote its length, 
$\bn_i$ its normal,  $\ell(i),r(i)$ the polygons to the left and right, 
respectively, and $\gamma_i$ be an arclength parameterization of $\Gamma_i$.
Finally we denote the union of the regions $\Omega_1,\dots,\Omega_n$ by $\Omega$ 
and denote the complement of $\Omega$ by $\Omega_0.$

Given positive constants $\mu_1,\dots,\mu_n$ and $\nu_1,\dots,\nu_n$ 
we consider the following boundary value problem
\begin{equation}
\label{eqn_BVP}
\begin{aligned}
  \Delta u_i &= 0\, \,\,\,\quad x \in \Omega_i,\,\,i=0,1,2,\dots,n,\\
  \mu_{\ell(i)} u_{\ell(i)} - \mu_{r(i)} u_{r(i)} &= f_i, \quad x \in \Gamma_i, \,\,i=1,\dots,m,\\
    \nu_{\ell(i)} \frac{\partial u_{\ell(i)}}{\partial n_i} - \nu_{r(i)}\frac{\partial u_{r(i)}}{\partial n_i} &= g_i, \quad x \in \Gamma_i, \,\,i=1,\dots,m,\\
    \lim_{|r|\to\infty} \left( r \log(r) u_0'(r) - u_0(r)\right)& = 0,
  \end{aligned}
\end{equation}
  where $f_i$ and $g_i$ are analytic functions on $\Gamma_i,$ $i=1,\dots,m$,
  and $\ell(i),r(i)$ denote the regions on the left and right with respect
  to the normal of edge $\Gamma_{i}$.

\begin{remark}
In this work we assume that all the normals $\bn_1,\dots,\bn_m$ to $\Gamma_1,\dots,\Gamma_m$ are
positively oriented with respect
to the parameterization $\gamma_{i}(t)$ 
of the edge $\Gamma_{i}$. 
Specifically, if $\Gamma_{i}$ is a line segment between vertices
$\bv_{\ell}$, $\bv_{r}$, and  $\gamma_{i}(t): [0,L_{i}] \to \Gamma_{i}$ is a parameterization
of $\Gamma_{i}$, given by
\begin{equation}
\gamma_{i}(t) = \bv_{\ell} + t\frac{\bv_{r}-\bv_{\ell}}{\|\bv_r-\bv_\ell \|} \, .
\end{equation}
Then the normal on edge $\Gamma_{i}$, is given by
\begin{equation}
\bn_{i} = \frac{(\bv_{r}-\bv_{\ell})^{\perp}}{\|\bv_r-\bv_\ell \|} \, ,
\end{equation}
where for a point $\bx=(x_{1},x_{2}) \in \mathbb{R}^{2}$,  
$\bx^{\perp} = (x_{2},-x_{1})$.
\end{remark}

\begin{figure}
\centering
{ \includegraphics[width=4cm]{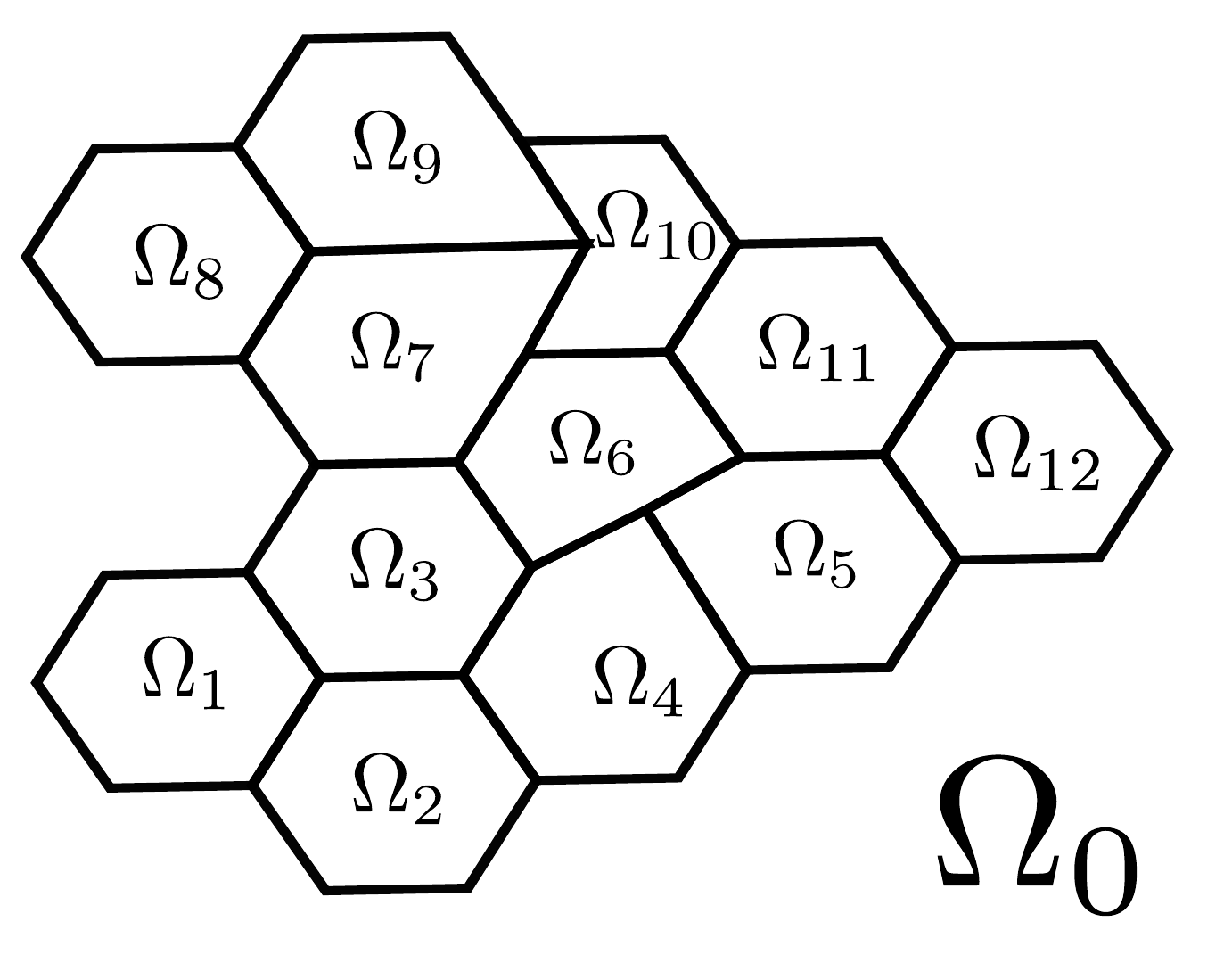} } \caption{Example of a composite region}\label{fig:comp_fig}
\end{figure}

\begin{remark}
The existence and uniqueness of solutions to (\ref{eqn_BVP}) 
is a classical result~\cite{mclean2000strongly}.
\end{remark}
\begin{remark}
In this paper we assume that no more than three edges meet at each vertex. 
Similar analysis holds for domains with higher-order junctions and will 
be published at a later date.
\end{remark}
\begin{remark}
Here we assume that $\mu_1,\dots,\mu_n$, and $\nu_1,\dots,\nu_n$ are 
positive constants. 
In principle the analysis presented here extends to the case where the 
constants are negative or complex provided the the constants $(\mu_{j} \nu_{i} + \mu_{i} \nu_{j})/(\mu_{j} \nu_{i} - \mu_{i} \nu_{j})$ across each edge are outside the closure of the essential spectrum of the double layer potential
defined on the boundary, and the underlying differential equation admits a unique solution.
Note that for non-negative coefficients this is always true, since these constants are all in magnitude greater than $1$, and
the spectral radius of the double layer potential is bounded by $1$.
\end{remark}

 \subsection{Layer potentials}
Before reducing the boundary value problem (\ref{eqn_BVP}) to a boundary integral equation we first introduce the layer potential operators and summarize their relevant properties.
\begin{definition}
Given a density $\sigma$ defined on $\Gamma_i,$ $i=1,\dots,m,$ the single-layer potential is defined by
\beqn
\cS_{\Gamma_i}[\sigma](\by) = -\frac{1}{2\pi} \int_{\Gamma_i} \log{\|\bx - \by\|} \sigma(\bx) dS_{\bx}
\eeqn
and the double-layer potential is defined via the formula
\beqn
\cD_{\Gamma_i}[\sigma](\by) = \frac{1}{2\pi} \int_{\Gamma_i} \frac{\bn(\bx) \cdot( \by-\bx)}{\|\bx-\by\|^2}\sigma(\bx) dS_{\bx},
\eeqn
\end{definition}

\begin{remark}
In light of the previous definition, evidently the adjoint of the 
double-layer potential is given by the formula
\beqn
\cD^*_{\Gamma_i}[\sigma](\by) = \frac{1}{2\pi} \int_{\Gamma_i} \frac{\bn(\by) \cdot( \bx-\by)}{\|\bx-\by\|^2}\sigma(\bx) dS_{\bx},
\eeqn
\end{remark}
\begin{definition}
For $\bx\in \Gamma$ we define the kernel $K(\bx,\by)$ by
\begin{equation}
K(\bx,\by) =\frac{1}{2\pi}  \frac{\bn(\bx) \cdot( \by-\bx)}{\|\bx-\by\|^2} \, .
\end{equation}
\end{definition}

The following theorems describe the limiting values of the single
and double layer potential on the boundary $\Gamma_{i}$.

\begin{theorem}\label{thm_reps}
Suppose that $\bx_{0}$ is a point in the interior of the segment 
$\Gamma_{i}$.
Suppose the point $\bx$ approaches a point $\bx_0$ 
along a path such that 
\begin{align}
-1+\alpha<\frac{\bx-\bx_0}{\|\bx-\bx_0\|} \cdot \gamma_i'(t_0) <1-\alpha
\end{align}
for some $\alpha >0$. 
If $(\bx - \bx_{0}) \cdot \bn_{i} < 0$,
we will refer to this limit as $\bx \to \bx_{0}^{-}$,
and if $(\bx-\bx_{0}) \cdot \bn_{i} > 0$, 
we will refer to this limit as $\bx \to \bx_{0}^{+}$.

Then
\begin{align}
&\lim_{\bx \to \bx_0^{\pm}} \cS_{\Gamma_i}[\sigma](\bx) =\, \cS_{\Gamma_i}[\sigma](\bx_0)\\
&\lim_{\bx \to \bx_0^{\pm}} \cD_{\Gamma_i}[\rho](\bx) = {\rm p.v.}\, \cD_{\Gamma_i}[\rho](\bx_0) \mp \frac{\rho(\bx_0)}{2}\\
&\lim_{\bx \to \bx_0^{\pm}} \bn_{i} \cdot \nabla \cS_{\Gamma_i}[\rho](\bx) ={\rm p.v.}\, \cD^*_{\Gamma_i}[\rho](\bx_0) \pm \frac{\rho(\bx_0)}{2},
\end{align}
where p.v. refers to the fact that the principal value of the integral 
should be taken.

Moreover, both the limits
\begin{align}
\lim_{\bx \to \bx_0^{\pm}} \bn_{i} \cdot \nabla \cD_{\Gamma_i}[\rho](\bx) ,
\end{align}
exist and are equal.
\end{theorem}
\begin{remark}
In the following we will suppress the ${\rm p.v.}$ from expressions 
involving layer potentials evaluated at a point on the boundary. 
Unless otherwise stated, in such cases the principal value should always 
be taken.
\end{remark}

\subsection{Integral representation}
In classical potential theory the boundary value problem (\ref{eqn_BVP}) 
is reduced to a boundary integral equation for a new collection of 
unknowns  $\rho_i, \sigma_i \in \mathbb{L}^2(\Gamma_i)$, $i=1,\dots,m$ 
related to $u_{i}: \Omega_{i} \to \mathbb{R}$, $i=1,\dots,n$ 
in the following manner
\beqn
\label{eq:udef}
u_{i}(\bx) = \frac{1}{\mu_{i}} \sum_{j=1}^m \cS_{\Gamma_{j}}[\rho_{j}](\bx) 
+ \frac{1}{\nu_{i}} \sum_{j=1}^m 
\cD_{\Gamma_{j}}[\sigma_{j}](\bx) \, \quad \bx \in \Omega_{i}.
\eeqn
We note that by construction $u_i$ is harmonic in $\Omega_i$, 
$i=0,1,\dots,n$. 
Enforcing the jump conditions across the edges and applying~\cref{thm_reps} 
yields the following system of integral equations for the unknown 
densities $\rho_i$ and $\sigma_i$
\begin{align}
-\frac{1}{2}
\sigma_{i} +\frac{\mu_{r(i)}\nu_{\ell(i)} - \mu_{\ell(i)}\nu_{r(i)}}{\mu_{r(i)}\nu_{\ell(i)} + \mu_{\ell(i)}\nu_{r(i)}}  \sum_{\ell=1}^m \cD_{\Gamma_{\ell}}[\sigma_{\ell}] &= \frac{\nu_{\ell(i)}\nu_{r(i)} f_{i}}{\mu_{r(i)}\nu_{\ell(i)} + \mu_{\ell(i)}\nu_{r(i)}} 
\label{eq:inteq1} \\
-\frac{1}{2}
\rho_{i} +\frac{\mu_{r(i)}\nu_{\ell(i)} - \mu_{\ell(i)}\nu_{r(i)}}{\mu_{r(i)}\nu_{\ell(i)} + \mu_{\ell(i)}\nu_{r(i)}}  \sum_{\ell=1}^m \cD^{\ast}_{\Gamma_{\ell}}[\rho_{\ell}] &= -\frac{\mu_{\ell(i)}\mu_{r(i)} g_{i}}{\mu_{r(i)}\nu_{\ell(i)} + \mu_{\ell(i)}\nu_{r(i)}} \, , \label{eq:inteq2}
\end{align}
 for $i=1,\dots,m.$

We note that the preceding representation has several advantages. 
Firstly, the kernels of integral equations (\ref{eq:inteq1}) and 
(\ref{eq:inteq2}) are smooth except at the vertices. 
In particular, the weakly-singular and hypersingular terms arising from 
the single-layer potential and the derivative of the double-layer 
potential, respectively, are absent. 
Secondly, the equations for the single-layer density $\rho$ and the 
double-layer density $\sigma$ are completely decoupled and can be analyzed 
separately. 
Moreover, (\ref{eq:inteq2}) is the adjoint of (\ref{eq:inteq1}) and hence 
the structure of solutions to (\ref{eq:inteq2}) can be inferred from the 
behavior of solutions to (\ref{eq:inteq1}). 

\begin{remark}
The above representation also appears in \cite{helsing2011effective} 
and is related to the 
work in \cite{greengard2012stable}. 
It has been shown in \cite{claeys2017second} that
the boundary integral 
equations (\ref{eq:inteq1}),(\ref{eq:inteq2}) are 
well-posed for $f_i,g_i \in \mathbb{L}^2[\Gamma_i]$. 
\end{remark}

\subsection{The single-vertex problem}
The following lemma reduces the problem of analyzing the behavior of 
the densities $\rho$ and $\sigma$ in the vicinity of a triple junction 
with locally-analytic data to the analysis of an integral equation on a 
set of three intersecting line segments.  
\begin{lemma}
Let $\sigma,\rho$ satisfy the boundary integral 
equation (\ref{eq:inteq1}) and (\ref{eq:inteq2}), respectively. 
Consider three edges $\Gamma_i,$ $\Gamma_j,$ and $\Gamma_k$ meeting at a vertex 
$v_p$. 
If $\bx_p$ denotes the coordinates of the vertex $v_m$ 
then there exists an $r>0$ such that
\begin{equation}
 \int_{\Gamma \setminus B_r(\bx_p)} K(\bx,\by) \sigma(\bx) \,dS_{\bx},\quad\quad  \int_{\Gamma \setminus B_r(\bx_p)} K(\by,\bx) \rho(\bx) \,dS_{\bx},
\end{equation}
are analytic functions of $\by$ for all $\by \in B_r(\bx_p)$. 
Here $B_r(\bx_p)$ denotes the ball of radius $r$ centered at $\bx_p$.
\end{lemma}
\begin{remark}
We note that by choosing $r$ sufficiently small we can assume that the 
intersection of all three-edges with $B_r(\bx_p)$ are of length $r$. 
Moreover, since Laplace's equation is invariant under scalings the 
subproblem associated with the corner can be mapped to an integral 
equation on three intersecting edges of unit length.
\end{remark}

In light of the preceding remark, in the remainder of this paper we 
restrict our attention to the geometry shown in Figure \ref{fig:trip_geom}.

\begin{figure}
\begin{center}
\begin{tikzpicture}
\path[draw,thick] (2.87,-0.089)--(4,4);
\path[draw,thick] (4,4)--(0.197,2.11838);
\path[draw,thick] (4,4)--(8.236,4.238);
\path[draw, dashed] (4,4) circle (1.5);

\node at (3.2,5.8) {$\theta_1$};
\node at (2.8,2.4) {$\theta_2$};
\node at (5.5,2.7) {$\theta_3$};
\node at (2.65,-0.3) {$\Gamma_{(2,3)}$};
\node at (0,1.9) {$\Gamma_{(1,2)}$};
\node at (8.84,4.24) {$\Gamma_{(3,1)}$};
\end{tikzpicture}
\end{center}\caption{Geometry near a triple junction}\label{fig:trip_geom}
\end{figure}
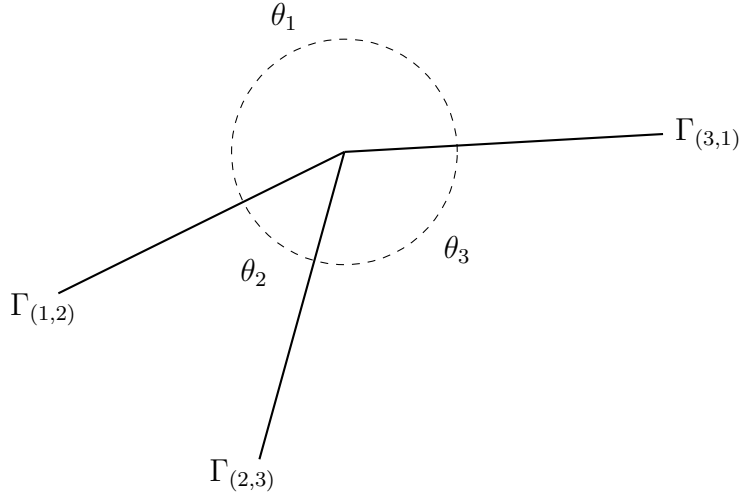

The following notation will be used in our analysis of triple junctions.
\begin{remark}
Suppose that $\Gamma_{(\ell,m)}$ and $\Gamma_{(\ell',m')}$ are two (possibly identical) edges of a triple junction in which all edges are of length one. For $(\ell,m)$ and $(\ell',m')$ in $\{(1,2),(2,3),(3,1)\}$ and $t \in (0,1)$ let
\begin{align}
\cD_{(\ell,m);(\ell',m')}[\sigma](t) ={\rm p.v.} \left.\cD_{\Gamma_{(\ell,m)}}[\sigma]\right|_{\Gamma_{(\ell',m')}}
\label{eq:dsingvertdef}
\end{align}
and
\begin{align}
\cD^*_{(\ell,m);(\ell',m')}[\rho](t) ={\rm p.v.} \left.\cD^*_{\Gamma_{(\ell,m)}}[\rho]\right|_{\Gamma_{(\ell',m')}}
\label{eq:dastsingvertdef}
\end{align} 
for any $\sigma,\rho \in \mathbb{L}^2(\Gamma_{(3,1)}\cup\Gamma_{(1,2)}\cup\Gamma_{(2,3)}).$ Note that if $(\ell,m) = (\ell',m')$ then both quantities are identically zero for any $\sigma$ and $\rho.$ If $(\ell,m) \neq (\ell',m')$ then the principal value is not required. 

Finally, in the following we will also denote the restrictions of $\sigma$ and $\rho$ to an edge $\Gamma_{(\ell,m)}$ by $\sigma_{(\ell,m)}$ and $\rho_{(\ell,m)},$ respectively.
\end{remark}


\section{Main results}
\label{sec:mainres}
In this section we state several theorems which characterize
the behavior of the solutions $\sigma,\rho$ to ~\cref{eq:inteq1,eq:inteq2} for 
the single-vertex problem with piecewise smooth boundary data $f$ and $g.$
 Before doing so we first introduce some convenient notation. To that end,
let $\Gamma_{(1,2)},\Gamma_{(2,3)}$ and $\Gamma_{(3,1)}$ be three edges 
of unit length meeting at a vertex as in Figure \ref{fig:trip_geom}. 
Let $\theta_1,$ $\theta_2,$ and $\theta_3$ be the 
angles at which they meet and suppose that $0<\theta_1,\theta_2,\theta_3<2\pi$ 
are real numbers summing to $2\pi$. 
Let $\Omega_1$ denote the region bordered by $\Gamma_{(3,1)}$ 
and $\Gamma_{(1,2)},$ $\Omega_2$ the region bordered by 
$\Gamma_{(1,2)}$ and $\Gamma_{(2,3)},$ and $\Omega_3$ 
denote the region bordered by $\Gamma_{(2,3)}$ and $\Gamma_{(3,1)}$.  
Finally, let $\mu_i$ and $\nu_i$ be the parameters corresponding to 
$\Omega_i,$ $i=1,2,3$ and define the constants $d_{(1,2)},d_{(2,3)}$ and 
$d_{(3,1)}$ by
\begin{equation}
\begin{aligned}
d_{(1,2)} &= \frac{\mu_1\nu_2-\mu_2\nu_1}{\mu_1\nu_2+\mu_2\nu_1},\\
d_{(2,3)} &= \frac{\mu_2\nu_3-\mu_3\nu_2}{\mu_2\nu_3+\mu_3\nu_2},\\
d_{(3,1)} &= \frac{\mu_3\nu_1-\mu_1\nu_3}{\mu_3\nu_1+\mu_1\nu_3}.
\end{aligned}
\label{eq:dijdef}
\end{equation}

\begin{remark}
\label{rem:abc}
We note the following properties of 
$d_{(3,1)},d_{(1,2)},d_{(2,3)}$ which, for notational convenience,
we will denote by $a,$ $b,$ and $c,$ respectively.
Firstly, since  $\mu_{i},\nu_{i}$ are positive
real numbers, it follows that
$a,b,c \in (-1,1)$. 
Secondly, a simple calculation shows that
$c = -(a+b)/(1+ab)$. 
Thus, at each triple junction, there are two parameters
$(a,b)$ which encapsulate the relevant information
regarding material properties at that junction. 
For the rest of the paper, in a slight abuse of notation, 
we will refer to $(a,b)$ as the material parameters.
\end{remark}

Next we define several quantities which will be used in the statement 
of the main results. Let $\mathcal{J}$ denote the set of 
indices $\{ (1,2),(2,3),(3,1) \}$ and 
$X= \mathbb{L}^{2} (\Gamma_{(1,2)}) \otimes
\mathbb{L}^{2}(\Gamma_{(2,3)}) \otimes \mathbb{L}^{2}(\Gamma_{(3,1)})$.
Let $\Kdir: X \to X$, and $\Kneu: X \to X$ denote
the bounded operators in~\cref{eq:inteq1} and~\cref{eq:inteq2}
respectively. 
For any operator $A: X\to X$, $\bh \in X$, and $(i,j) \in \mathcal{J}$, 
we denote the restriction of $A[\bh]$ 
to the edge $\Gamma_{(i,j)}$ by $A[\bh]_{(i,j)}$.
For example, given $\bh(t) = [h_{(1,2)}(t), h_{(2,3)}(t), h_{(3,1)}(t)]^{T} 
\in X$, and $(i,j)\in \mathcal{J}$,  
\begin{equation}
\Kdir[\bh]_{(i,j)} =  -\frac{1}{2} h_{(i,j)} + d_{(i,j)} 
\sum_{(\ell,m) \in \mathcal{J}} \mathcal{D}_{(\ell,m);(i,j)}[h_{(\ell,m)}] \, , 
\end{equation}
where the operators $\cD_{(\ell,m);(i,j)}$ are defined 
in~\cref{eq:dsingvertdef}.

We are interested in the following two problems:
\begin{enumerate}
\item for what collection of $\bh \in X$ are $\Kdir[\bh]$, 
and $\Kneu[\bh]$ piecewise smooth functions on each of the edges
$\Gamma_{(i,j)}$, $(i,j) \in \mathcal{J}$; and
\item  given 
$h_{(i,j)} \in \mathcal{P}_{N}$, a polynomial of degree at most $N$, construct an 
explicit basis for $\Kdir^{-1}[\bh]$ and $\Kneu^{-1}[\bh]$ \, .
\end{enumerate}
In~\cref{subsec:dirich}, we address these questions for $\Kdir$, while 
in~\cref{subsec:neu} we present analogous results for $\Kneu$.

\subsection{Analysis of $\Kdir$ \label{subsec:dirich}}
Suppose that $\bh(t) = [h_{(1,2)}(t),h_{(2,3)}(t),h_{(3,1)}(t)]^{T} 
= \bv t^{\beta}$, 
where $t$ denotes the distance along the edge $\Gamma_{(i,j)}$ 
from the triple junction, and
$\bv \in \mathbb{R}^{3}$ and $\beta \in \mathbb{R}$ are constants.

In the following theorem, we derive necessary conditions
on $\beta,\bv$ such that $\Kdir[\bh]_{(i,j)}$ is a smooth
function on each edge $\Gamma_{(i,j)}$, $(i,j) \in \mathcal{J}$.
\begin{theorem}\label{thm_fwd}
Let $\Adir(a,b,\beta) \in \mathbb{R}^{3\times 3}$ denote the matrix
given by
\begin{equation}
\label{eq:defAdir}
\Adir(a,b,\beta) = 
\begin{pmatrix}
\spt{\beta} & b \sin{\beta(\pi-\theta_2)}& -b\sin{\beta (\pi-\theta_1)}\\
(a+b)/(1+ab) \sin{\beta (\pi-\theta_2)} & \spt{\beta} &
-(a+b)/(1+ab)\sin{\beta(\pi-\theta_3)} \\  
a \sin{\pi \beta(1-\theta_1)} & -a\sin{\pi\beta (1-\theta_3)} & \spt{\beta} 
\end{pmatrix} \, .
\end{equation}
Suppose that $\beta$ is a positive real number such that 
$\det{\Adir(d_{(3,1)},d_{(1,2)}, \beta)} =0$
and that $\bv$ is a null-vector of $\Adir(d_{(3,1)},d_{(1,2)},\beta)$.
Let $\bh(t) = \bv t^{\beta}$, $0<t<1$.  
Then $\Kdir[\bh]_{(i,j)}$ is an analytic function of $t$, for $0<t<1$, 
on each of the edges $\Gamma_{(i,j)}$, $(i,j) \in \mathcal{J}$.
\end{theorem}

The above theorem guarantees that for appropriately 
chosen densities $\bh \in X$, 
the potential $\Kdir[\bh]$ 
is an analytic function on each of the edges.  

We now consider the construction of a
basis for 
$\Kdir^{-1}[\bh]$, when $h_{(i,j)} \in \mathcal{P}_{N}$,
$(i,j) \in \mathcal{J}$ for some $N>0$.

In order to prove this result, we require a collection 
of $\beta, \bv$ satisfying the conditions of~\cref{thm_fwd}.
The following lemma states the existence of a countable
collection of $\beta,\bv$ which are analytic on a subset of 
$(-1,1)^2$.
\begin{lemma}
\label{lem:betavs_existence_dir}
Suppose that $\theta_{1},\theta_{2},\theta_{3}$ are irrational numbers
summing to $2\pi$,
and $(a,b) \in (-1,1)^2$.
Then there exists a countable collection of open subsets of $(-1,1)^2,$ 
denoted by $S_{i,j},$ as well as a corresponding set of functions 
$\beta_{i,j}: S_{i,j} \to \mathbb{R}$, 
$i=0,1,2,\ldots$, $j=0,1,2$ such that
$\det{\Adir}(a,b,\beta_{i,j}) = 0$ for all $(a,b) \in S_{i,j}$.
The corresponding 
null-vectors $\bv_{i,j}: S_{i,j} \to \mathbb{R}^{3}$ of 
$\Adir(a,b,\beta_{i,j})$ are also analytic functions.
Finally, for any $N>0$, $|\cap_{i=0}^{N}\cap_{j=0}^{2} S_{i,j} | > 0$.   
\end{lemma}

In the following theorem, we present the main result of this section
which gives a basis for $\Kdir^{-1}[\bh]$. 
\begin{theorem}
\label{thm_inv}
Consider the same geometry as in~\cref{fig:trip_geom},
where $\theta_{1},\theta_{2}$, and
$\theta_{3}$ sum to $2\pi$ and $\theta_1/\pi,$ $\theta_2/\pi,$ and $\theta_3/\pi$
 are irrational. 
Let $\beta_{i,j},\bv_{i,j},S_{i,j}$, $i=0,1,2,\ldots$, $j=0,1,2$ 
be as defined in~\cref{lem:betavs_existence_dir}, and for any
positive integer $N$, let $S_{N}$ denote the region of common 
analyticity of 
$\beta_{i,j},\bv_{i,j}$, i.e., 
$S_{N} = \cap_{i=0}^{N} \cap_{j=0}^{2} S_{i,j}$.
Finally, suppose that $h_{(i,j)}^{k}$, 
$(i,j) \in \mathcal{J}, k=0,1,2,\ldots N$ 
are real constants, and
define $h_{(i,j)}$ by
\begin{equation}
\begin{aligned}
h_{(i,j)}(t) &= \sum_{k=0}^{N} h_{(i,j)}^{k} t^{k} \, , \end{aligned}
\label{eq:def_h}  
\end{equation}
$0<t<$1.
Then there exists an open region $\tilde{S}_{N} \subset S_{N} \subset 
(-1,1)^2$ 
with $|\tilde{S}_{N}|>0$ such that the following holds.
For all $(a,b) \in \tilde{S}_{N}$, 
there exist constants $p_{i,j}$, $i=0,1,\ldots N$, $j=0,1,2$,
such that
\begin{equation}
\sigma = 
\begin{bmatrix}
\sigma_{1,2}(t) \\
\sigma_{2,3}(t) \\
\sigma_{3,1}(t)
\end{bmatrix} = \sum_{i=0}^{N} \sum_{j=0}^{2} 
p_{i,j} \bv_{i,j} t^{\beta_{i,j}}
\, ,
\end{equation}
satisfies
\begin{equation}
\max_{(i,j)\in \mathcal{J}} 
\left| h_{(i,j)} - \Kdir[\sigma]_{(i,j)} \right|
\leq C t^{N+1} \, ,
\end{equation}
for $0<t<1$, where $C$ is a constant.
\end{theorem}

\subsection{Analysis of $\Kneu$ \label{subsec:neu}}
Suppose that $\bh(t) = [h_{(1,2)}(t),h_{(2,3)}(t),h_{(3,1)}(t)]^{T} = 
\bw t^{\beta-1}$,  
where $t$ denotes the distance on the edge $\Gamma_{(i,j)}$ 
from the triple junction, and $\bw \in \mathbb{R}^{3}$ 
and $\beta$ are constants.
In the following theorem, we discuss necessary conditions
on $\beta,\bw$ guaranteeing that $\Kneu[\bh]_{(i,j)}$ is a smooth
function on each edge $\Gamma_{(i,j)}$, $(i,j) \in \mathcal{J}$.
\begin{theorem}\label{thm_fwd_neu}
Let $\Aneu(a,b,\beta) \in \mathbb{R}^{3\times 3}$ denote the matrix
given by
\begin{equation}
\label{eq:defAneu}
\Aneu(a,b,\beta) = 
\begin{pmatrix}
\spt{\beta} & -b \sin{\beta(\pi-\theta_2)}& b\sin{\beta (\pi-\theta_1)}\\
-(a+b)/(1+ab) \sin{\beta (\pi-\theta_2)} & \spt{\beta} &
(a+b)/(1+ab)\sin{\beta(\pi-\theta_3)} \\  
-a \sin{\beta(\pi-\theta_1)} & a\sin{\pi\beta (1-\theta_3)} & \spt{\beta} 
\end{pmatrix} \, .
\end{equation}
Suppose that $\beta$ is a positive real number such that 
$\det{\Aneu(d_{(3,1)},d_{(1,2)}, \beta)} =0$
and let $\bw$ denote a corresponding null-vector of 
$\Aneu(d_{(3,1)},d_{(1,2)},\beta)$.
Let $h = \bw t^{\beta-1}$, $0<t<1$.  
Then $\Kneu[h]_{(i,j)}$ is an analytic function of $t$, for $0<t<1$, 
on each of the edges $\Gamma_{(i,j)}$, $(i,j) \in \mathcal{J}$.
\end{theorem}

Before proceeding a few remarks are in order.
\begin{remark}
We note that $\det{\Adir(a,b,\beta)} = \det{\Aneu(a,b,\beta)}$.
Thus, the existence of $\beta,\bw$, which satisfy the conditions
of~\cref{thm_fwd_neu} is guaranteed by~\cref{lem:betavs_existence_dir}.
\end{remark}

\begin{remark}
For a given $\beta$, if there exists a $\bv \in \mathbb{R}^{3}$ 
such that $\Kdir[\bv t^{\beta}]$ is piecewise smooth then
there also exists a vector $\bw\in \mathbb{R}^{3}$ such that 
$\Kneu[\bw t^{\beta-1}]$ is also a smooth function.
However, the requirement that $\bw t^{\beta-1} \in X$ 
implies that for $\Kneu$, only $\beta$'s which satisfy
$\beta > 1/2$ are admissible.

For $\Kdir$, note that 
$\beta_{0,j} = 0$ for $j=0,1,2$ (see proof of~\cref{lem:betavs_existence_dir}
contained in~\cref{subsec:beta_v_existence}).
These densities are essential for the proof of~\cref{thm_inv}, 
since these are the only basis functions for which
the projection of their image under  
$\Kdir$ onto the constant functions are non-zero.

However, since $\beta_{0,j} \not > 1/2$, 
the densities $\bw_{0,j} t^{\beta_{0,j}-1}$ 
are excluded from the representation for the solution
to the equation $\Kneu[\sigma] = \bh$. 
Note that, unlike $\Kdir[\bv_{i,j} t^{\beta_{i,j}}]$, 
$\Kneu[\bw_{i,j} t^{\beta_{i,j}-1}]$, $i=1,2,\ldots$, $j=0,1,2$,
have a non-zero projection onto the constants (see~\cref{lem:pot_fwd_neu}).
\end{remark}

The following theorem is a converse of~\cref{thm_fwd_neu} under suitable restrictions.
\begin{theorem}
\label{thm_neu_inv}
Consider the same geometry as in~\cref{fig:trip_geom},
where $\theta_{1},\theta_{2}$, and
$\theta_{3}$ are irrational numbers summing to $2.$ 
Let $\beta_{i,j},\bw_{i,j},S_{i,j}$, $i=0,1,2,\ldots$, $j=0,1,2$ 
be as defined in~\cref{lem:betavs_existence_dir}. Let $T_{i,j}$ denote the 
open subset of $(-1,1)^2$ on which $\beta_{i,j}$ and $\bw_{i,j}$ are analytic and 
$\beta_{i,j} >1/2.$
 For any positive integer $N$, let $\snneu$ denote the region of common 
analyticity of 
$\beta_{i,j},\bw_{i,j}$, i.e., 
$\snneu = \cap_{i=1}^{N+1} \cap_{j=0}^{2} T_{i,j}$.
Finally, suppose that $h_{(i,j)}^{k}$, 
$(i,j) \in \mathcal{J}, k=0,1,2,\ldots N$ 
are real constants, and
define $h_{(i,j)}$ by
\begin{equation}
\begin{aligned}
h_{(i,j)}(t) &= \sum_{k=0}^{N} h_{(i,j)}^{k} t^{k} \, , \end{aligned}
\label{eq:def_h2}  
\end{equation}
$0<t<$1.

Then there exists an open region $\tsnneu \subset \snneu \subset 
(-1,1)^2$ 
with $|\tsnneu|>0$ such that the following holds.
For all $(a,b) \in \tsnneu$, 
there exist constants $p_{i,j}$, $i=1,2,\ldots N+1$, $j=0,1,2$,
such that
\begin{equation}
\sigma = 
\begin{bmatrix}
\sigma_{1,2}(t) \\
\sigma_{2,3}(t) \\
\sigma_{3,1}(t)
\end{bmatrix} = \sum_{i=1}^{N+1} \sum_{j=0}^{2} 
p_{i,j} \bw_{i,j} t^{\beta_{i,j}-1}
\, ,
\end{equation}
satisfies
\begin{equation}
\max_{(i,j)\in \mathcal{J}} 
\left| h_{(i,j)} - \Kneu[\sigma]_{(i,j)} \right|
\leq C t^{N+1} \, ,
\end{equation}
for $0<t<1$, where $C$ is a constant.
\end{theorem}


\section{Conjectures \label{sec:conj}}
There are four independent parameters 
that completely describe the triple junction problem, 
any two out of the three angles $\{ \theta_{1}, \theta_2, \theta_3 \},$ and any two of 
the parameters
$\{ d_{(1,2)}, d_{(2,3)}, d_{(3,1)} \}=\{b,c,a\}.$
Let $Y \subset \mathbb{R}^{4}$, denote the subset of $\mathbb{R}^{4}$
associated with the four free parameters 
that completely describe any triple junction given by
\begin{equation}
Y = \{ (\theta_{1},\theta_{2}, a, b) \, : \quad 0<\theta_{1}, \theta_{2}< 2\pi \, ,
\quad \theta_{1}+\theta_{2} < 2 \pi \, , \quad -1 < a,b < 1 \, \} \, .
\end{equation}
When $\theta_{1},\theta_{2}$, are irrational multiples of $\pi$,
and $(a,b)$ are in the neighborhoods of $a=0$, $b=0$, and $c=0$,
the results~\cref{thm_inv,thm_neu_inv}  
construct an explicit basis of non-smooth functions 
for the solutions of $\Kdir[\sigma] = \bh$,
and $\Kneu[\sigma] = \bh$, and show that this basis maps onto 
the space of boundary data given by piecewise polynomials on
each of the edges meeting at the triple junction.
However, extensive numerical studies suggest that both of these
results can be improved significantly. 
In particular, we believe that this analysis extends to 
all $(\theta_{1},\theta_{2},a,b) \in Y$, except for a set of measure
zero. 
Moreover, on the measure zero set where this basis is not sufficient,
we expect the solution to have additional logarithmic singularities;
including functions of the form $t^{\beta} \log{(t)} \bv$
should be sufficient to fix the deficiency of the basis.
We expect the analysis to be similar in spirit 
to the analysis carried out for the solution of Dirichlet
and Neumann problems for Laplace's equations
on vicinity of corners (see~\cite{serkh1,serkh2}).

In this section, we present a few open questions for further
extending the results~\cref{thm_inv,thm_neu_inv}, and present
numerical evidence to support these conjectures.

\subsection{Existence of $\beta_{i,j}$}
The solutions $\beta_{i,j}$, $i=0,1,2,\ldots$, 
$j=0,1,2$, are constructed as the implicit solutions of
$\det{\Adir(a,b,\beta)} = 0$ (recall that 
$\det{\Adir(a,b,\beta)} = \det{\Aneu(a,b,\beta)}$).
Note that $\det{\Adir(a,b,\beta)} = \sinb \cdot \alpha(a,b,c;\beta)$
where $\alpha$ is as defined in~\cref{eq:defalpha}.
From this, it follows that
$\beta_{i,0} = i$ 
always satisfies $\det{\Adir(a,b,\beta)} =0$ 
for all $\theta_{1},\theta_{2}$, and that
$\beta_{0,j} = 0$ results in three linearly independent basis
functions of the form $t^{\beta} \bv$ since $\Adir(a,b,0) = 0$.

The remaining $\beta_{i,j}$, $i=1,2,\ldots$, $j=1,2$, are constructed
in the following manner.
$\alpha(a,b,c;\beta)$ simplifies significantly along
$a=0$, $b=0$, and $c=0$, and the existence of $\beta_{i,j}$
which satisfy $\det{\Adir(a,b,\beta)}=0$ is guaranteed
based on the explicit construction detailed in~\cite{hoskins2018numerical}.
The construction then uses the implicit function theorem to extend
the existence of $\beta_{i,j}$ to a subset of $(a,b) \in (-1,1)^2$.
The implicit function theorem is a local result and 
only guarantees existence in local neighborhoods of the initial points.
However, extensive numerical evidence suggests that the $\beta_{i,j}$
are well-defined and analytic for all $(a,b) \in (-1,1)^2$ 
and all $\theta_{1},\theta_{2}$. 
In~\cref{fig:betas}, we plot a few of these functions to 
illustrate this result.

\begin{conjecture}
\label{conj1}
There exists a countable collection of $\beta_{i,j}$,
$i=1,2,\ldots,$, $j=1,2$ which satisfy
$\alpha(a,b,c;\beta_{i,j})=0$. 
Moreover, these $\beta_{i,j}$ are analytic functions
of $\theta_{1},\theta_{2},a$, and $b$, 
for all $(\theta_{1},\theta_{2},a,b) \in Y$.
\end{conjecture}

An alternate strategy for proving this result is by making the 
following observation.
For fixed $\theta_{1},\theta_{2}$, 
consider the curve $\gamma_{m}: (m,m+1) \to \mathbb{R}^3$ defined by
\begin{equation}
\gamma_{m}(\beta):= \frac{1}{\spt{\beta}}
\left( \sin{\beta(\pi-\theta_{2})}, \sin{\beta(\pi-\theta_{3})}, 
\sin{\beta(\pi-\theta_{1})} \right) \, ,
\end{equation}
where $m$ is an integer.
This defines a curve in $\mathbb{R}^{3}$ for which $|\gamma_{m}| 
\to \infty$ for each $m$.
Then consider the family of hyperboloids parameterized by $(a,b)$ 
given by
\begin{equation}
H(x,y,z;a,b) := -b(a+b)/(1+ab) x^2 -a(a+b)/(1+ab) y^2 + ab z^2 +1=0
\end{equation}
It follows immediately that the solutions to $\alpha(a,b,c;\beta) = 0$ 
can be characterized geometrically as points in the intersection of the
hyperboloid $H(x,y,z;a,b)$ with the curve $\gamma_{m}$.

\subsection{Completeness of the singular basis}
Having identified the $\beta_{i,j}$, and the corresponding null
vectors $\bv_{i,j}$ for $\Adir$, and $\bw_{i,j}$
for $\Aneu$, the 
second part of the proof shows that every set of boundary
data which is a polynomial of degree less than or equal to 
$N$ on each of the edges, has a solution to the integral
equations~\cref{eq:inteq1,eq:inteq2} in the 
$\bv_{i,j}t^{\beta_{i,j}}$ basis for $\Kdir$ 
and $\bw_{i,j}t^{\beta_{i,j}-1}$ for $\Kneu$
which agrees with the boundary data with error 
$O(t^{N+1})$.

This part of the proof relies on constructing an explicit mapping
from the coefficients of the density $\sigma$ 
in the $\bv_{i,j} t^{\beta_{i,j}}$  
to the coefficients of Taylor expansions for $\Kdir[\sigma].$
Then, along $a=0$, $b=0$, or $c=0$, based on the 
results in~\cite{hoskins2018numerical},
we show that this mapping is invertible along these edges.
It then follows from the continuity of determinants that the mapping
is invertible for open neighborhoods of the line segments $a=0$, $b=0$,
$c=0$. 
This implies that in the basis $\bv_{i,j} t^{\beta_{i,j}}$
there exists a $\sigma$  
such that  $|\Kdir[\sigma] - \bh| \leq O(t^{N+1})$,
for all boundary data $f$ 
in the space of polynomials with degree less than or
equal to $N$.

While we prove this result for an open neighborhood $(a,b)$
of the line segments $a=0$, $b=0$, $c=0$, when the angles
$\theta_{1},\theta_{2}$ are irrational multiples of $\pi$, 
we expect the bases to have this property
for all $(\theta_{1},\theta_{2},a,b) \in Y$ except for a measure
zero set. 
Moreover, this measure zero set is the set 
of $(\theta_{1},\theta_{2},a,b)$
for which the multiplicity of $\beta_{i,j}$ as a repeated root of 
$\det{\Adir(a,b,\beta_{i,j})} = 0$ is not the same as
the dimension of the null space of 
$\Adir(a,b,\beta_{i,j})$.

\begin{conjecture}
\label{conj:thm_inv}
Suppose that~\cref{conj1} holds, i.e.
$\beta_{i,j}: Y \to \mathbb{R}$ are analytic functions. 
Suppose further that $h_{(i,j)}^{k}$, 
$(i,j) \in \mathcal{J}, k=0,1,2,\ldots N$ 
are real constants, and
suppose that 
\begin{equation}
\begin{aligned}
h_{(i,j)}(t) &= \sum_{k=0}^{N} h_{(i,j)}^{k} t^{k} \, , \end{aligned}
\label{eq:def_h_conj}  
\end{equation}
$0<t<$1.
Then there exists a measure zero set $S$ such that for 
all $(\theta_{1},\theta_{2},a,b) \in Y \setminus S$ 
the following result holds.
There exist constants $p_{i,j}$, $i=0,1,\ldots N$, $j=0,1,2$,
such that
\begin{equation}
\sigma = 
\begin{bmatrix}
\sigma_{1,2}(t) \\
\sigma_{2,3}(t) \\
\sigma_{3,1}(t)
\end{bmatrix} = \sum_{i=0}^{N} \sum_{j=0}^{2} 
p_{i,j} \bv_{i,j} t^{\beta_{i,j}}
\, ,
\end{equation}
satisfies
\begin{equation}
\max_{(i,j)\in \mathcal{J}} 
\left| h_{(i,j)} - \Kdir[\sigma]_{(i,j)} \right|
\leq C t^{N+1} \, ,
\end{equation}
for $0<t<1$, where $C$ is a constant.
\end{conjecture}

In~\cref{fig:betas}, we plot sections of the zero measure set 
on which~\cref{conj:thm_inv} does not hold.

\begin{figure}
\begin{center}
\includegraphics[width=\linewidth]{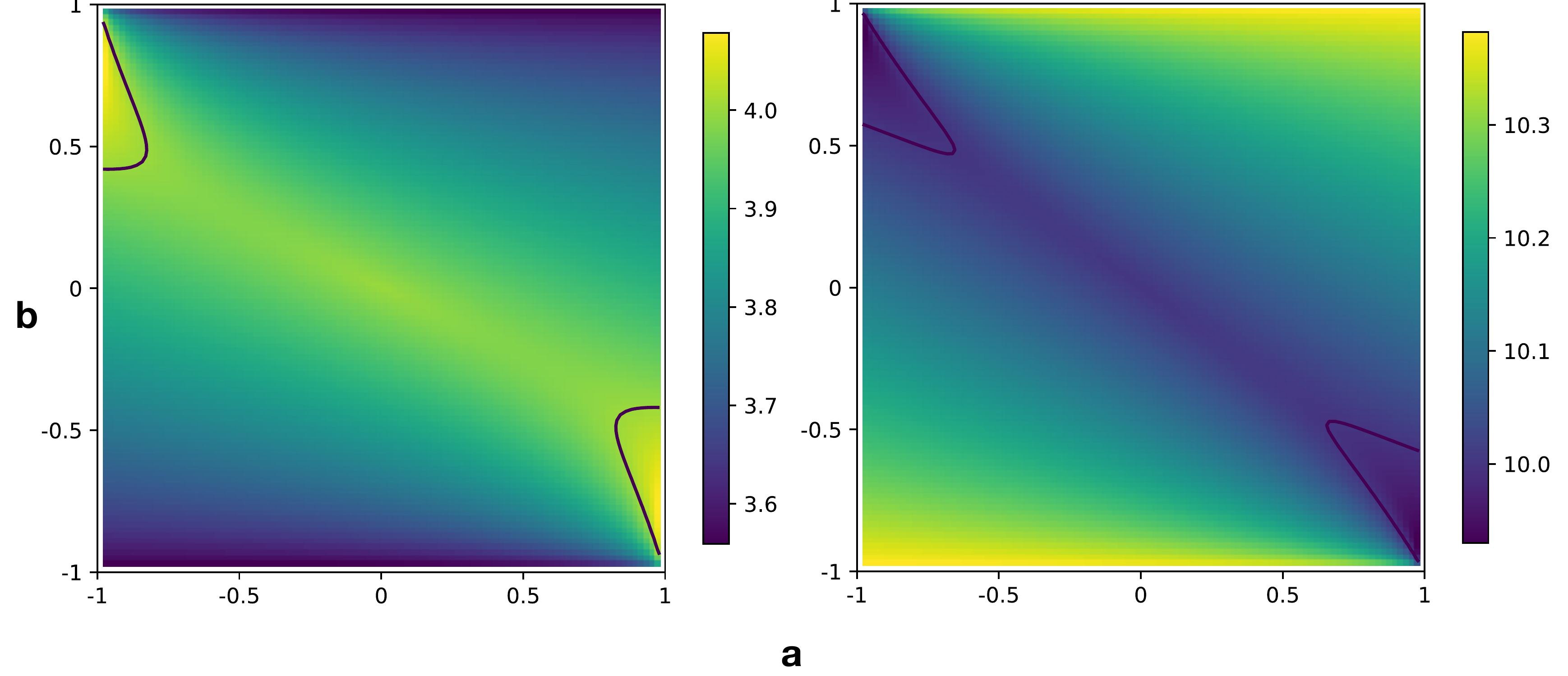}
\end{center}
\caption{Plots for $\beta(a,b)$ 
which satisfy $\det{\Adir(a,b,\beta)}=0$ at a triple
junction with angles $\theta_{1} = \pi/\sqrt{2}$, $\theta_{2} = \pi/\sqrt{3}$. 
$\beta(0,0)=4$ for the figure on the left, and $\beta(0,0) = 10$ for figure
on the right. In both of the figures, the solid black lines indicate sections
of the conjectured
measure zero set $S$ defined in~\cref{conj:thm_inv}.} 
\label{fig:betas}
\end{figure}


\section{Discretization of~\cref{eq:inteq1,eq:inteq2}\label{sec:numdis}}
In this section we discuss a numerical method for solving~\cref{eq:inteq1,eq:inteq2} 
for the unknown 
densities $\sigma,\rho$ which exploits the analysis of
 their behavior in the vicinity of triple
junctions.
There are two general approaches for discretizing these
integral equations:
Galerkin methods, in which the densities $\rho$ and $\sigma$ 
is represented directly in terms of appropriate basis functions,
and Nystr\"{o}m methods, where the solution is represented
in terms of its values at specially chosen discretization nodes.
In this paper, we use a Nystr\"{o}m discretization 
for solving~\cref{eq:inteq1}, though we note that the 
expansions in~\cref{thm_inv,thm_neu_inv}
can also be used to construct efficient Galerkin discretizations. 

In~\cite{bremer2010universal}, the authors developed a Nystr\"{o}m discretization for resolving the singular behavior of solutions to integral equations
in the vicinity of corners. 
In this approach, the authors obtain a basis of solutions to the integral equation in the vicinity of the corner
by solving a small number of local problems.
Based on these families of solutions, discretization nodes capable of interpolating the span of these solutions, coupled with quadratures for handling far-field interactions (inner products of the basis of solutions with smooth functions), and
special quadratures for handling near interactions (for resolving the near singular behavior of the kernel in the vicinity of the corner) are developed. 
This approach was later specialized for the solution of Laplace's equation on polygonal domains to obtain universal discretization
nodes, and quadrature rules~\cite{bremer2010efficient}. 

Recent advances in the analysis of integral equations for Laplace's equation have
provided analytic representations of solutions to integral equations
in the vicinity of the corners~\cite{serkh1,serkh2}, obviating the need for obtaining the span of solutions in the vicinity of corners through numerical means. Based on the approach above, these analytical results have been exploited to construct universal discretization and quadratures for solutions in vicinity of corners~\cite{jeremydis}. 
We briefly discuss the construction of the Nystr\"{o}m discretization in~\cite{jeremydis} below.
Let $\mathcal{F}$ denote the family of functions
\begin{equation}
\mathcal{F} = \{t^{\beta} \quad \text{for all } \beta \in \{0\} \cup [1/2,50],
\quad 0<t<1 \} \, .
\end{equation}
Then there exist $t_{j} \in [0,1]$, $w_{j}>0$, an orthogonal basis $\phi_{j}(t)$, $j=1,2,\ldots \kjer=36$, and a $\kjer \times \kjer$ matrix $V$ whose condition number is $O(1)$, with the following features.
For any $f\in \mathcal{F}$, there exists $c_{j}$, such that 
\begin{equation}
\left|f(t) - \sum_{j=1}^{\kjer} c_{j} \phi_{j}(t) \right|_{\bL^{2}[0,1]} < \eps \, .
\end{equation}
Let $f_{j}=f(t_{j}) \sqrt{w_{j}}$ denote the samples of the function at the discretization nodes scaled by the square root of 
the quadrature weights. 
The matrix $V$ maps $f_{j}$ to its coefficients $c_{j}$ in the $\phi_{j}$ basis.
Finally the weights $w_{j}$ are such that
\begin{equation}
\left|\int_{0}^{1} f(t) dt - \sum_{j=1}^{\kjer} f_{j} \sqrt{w_{j}} \right| =\left|\int_{0}^{1} f(t) dt - \sum_{j=1}^{\kjer} f(t_{j}) w_{j} \right|  \leq \eps \, .
\end{equation}
Specialized quadrature rules for handling the near-singular interaction between corner panels which
meet at the same vertex are also constructed.
The Dirichlet problem for Laplace's equation can then be discretized using panels
with scaled Gauss-Legendre nodes for panels which are away from corners, and using scaled nodes $t_{j}$
for panels at corners.
 
In the vicinity of triple junctions, the behavior of the solution $\sigma$ of~\cref{eq:inteq1} 
can be represented to high-order as a linear combination of functions in $\mathcal{F}$. 
Thus the discretization for the Dirichlet problem discussed above can be used to obtain a Nystr\"{o}m discretization
for~\cref{eq:inteq1}. Unfortunately, the same is not true when solving~\cref{eq:inteq2},
since the singular behavior
of $\rho$ is not contained in the span of $\mathcal{F}$.
In particular, the leading order singularity
in $\rho$ is of the form $t^{\beta}$ 
where $\beta \in (-\frac{1}{2},0)$.
The nature of the singularity of $\rho$ is similar
to the singular behavior of solutions
to integral equations corresponding to the Neumann problem
on polygonal domains.

Recall that ~\cref{eq:inteq2} is 
the adjoint of~\cref{eq:inteq1}.
Thus, formally, one could use the transpose
of the Nystr\"{o}m discretization of~\cref{eq:inteq1} to
solve~\cref{eq:inteq2}. Specifically, if
 $\overline{\rho} = \{\rho_{j}\}_{j=1}^{N}$ are 
the unknown values of $\rho$ at the discretization nodes, 
and $\overline{g} = \{g_{j}\}_{j=1}^{N}$ 
denote the samples of the boundary data for~\cref{eq:inteq2}
at the discretization nodes, then
we solve the linear system 
\begin{equation}
M^{T} \overline{\rho} = \overline{g} \, , 
\end{equation}
where $M$ is the matrix corresponding to
Nystr\"{o}m discretization of~\cref{eq:inteq1}.
The solution $\overline{\rho}$ is a high-order accurate
weak solution for the density $\rho$ which can be used
to evaluate the solution to~\cref{eq:inteq2} 
accurately away from the corner panels of the boundary $\Gamma$. 
This weak solution can be further refined to
obtain accurate approximations of the potentials in 
the vicinity of corner panels through solving a sequence
small linear systems for updating the solution $\rho_{j}$
in the vicinity of the corner panels.
This procedure is discussed in detail in~\cite{hoskins2020discretization}.

\section{Numerical examples \label{sec:numres}}
We illustrate the performance of the algorithm with
several numerical examples. 
In each of the problems let $\Omega_{0}$ denote the exterior
domain and $\Omega_{i}, i=1,2,\ldots N_{r}$ denote the interior
regions.
Let $c_{j,k}$, $k=1,2,\ldots 10$, denote points outside of the region
$\Omega_{j}$ for $j=1,2,\ldots N_{r}$.
The results in~\cref{subsec:cond1,subsec:cond2} have been 
computed using dense linear algebra routines, while the results in~\cref{subsec:acc,subsec:app}
have been computed using GMRES where the matrix vector product computation
has been accelerated using fast multipole methods~\cite{greengard1986fast}.

\subsection{Accuracy\label{subsec:acc}}
In order to demonstrate the accuracy of our method 
we solve the PDE with boundary data corresponding to known
harmonic functions using our discretization of the integral 
equation formulation.
We set $u_{j}(x) = \sum_{k=1}^{10} \log |x - c_{j,k}|$ and set
$u\equiv 0$ for $x \in \Omega_{0}$. 
We then compute the boundary data
\begin{equation}
f_{i} = \mu_{\ell(i)} u_{\ell(i)} - \mu_{r(i)} u_{r(i)} \, ,\quad
g_{i} = \mu_{\ell(i)} \frac{\partial u_{\ell(i)}}{\partial n} - \mu_{r(i)} \frac{\partial u_{r(i)}}{\partial n} \, ,
\end{equation}
and solve for $\sigma,\rho$. 
Given the discrete solution for $\sigma,\rho$, we compare the computed solution 
and plot the error in the computed
at targets in the interior of each of the regions.
In~\cref{fig:2tri,fig:4dia}, we demonstrate the results for two sample
geometries.
\begin{remark}
Note that we do not use special quadratures for handling near boundary targets which is responsible
for the loss of accuracy close to the boundary. 
For panels away from the corner, the potential at near boundary targets 
can be computed accurately using several standard methods such as 
Quadrature by expansion, or product quadrature (see~\cite{qbxquadrature,helsingnear,barnettnear}).
In order to evaluate the solution at points lying close to a corner panel a different 
approach is required. A detailed description of a computationally efficient algorithm for evaluating the
solution accurately arbitrarily close to a corner is presented in~\cite{hoskins2020discretization}.
\end{remark}

\begin{figure}
\begin{center}
\includegraphics[width=\linewidth]{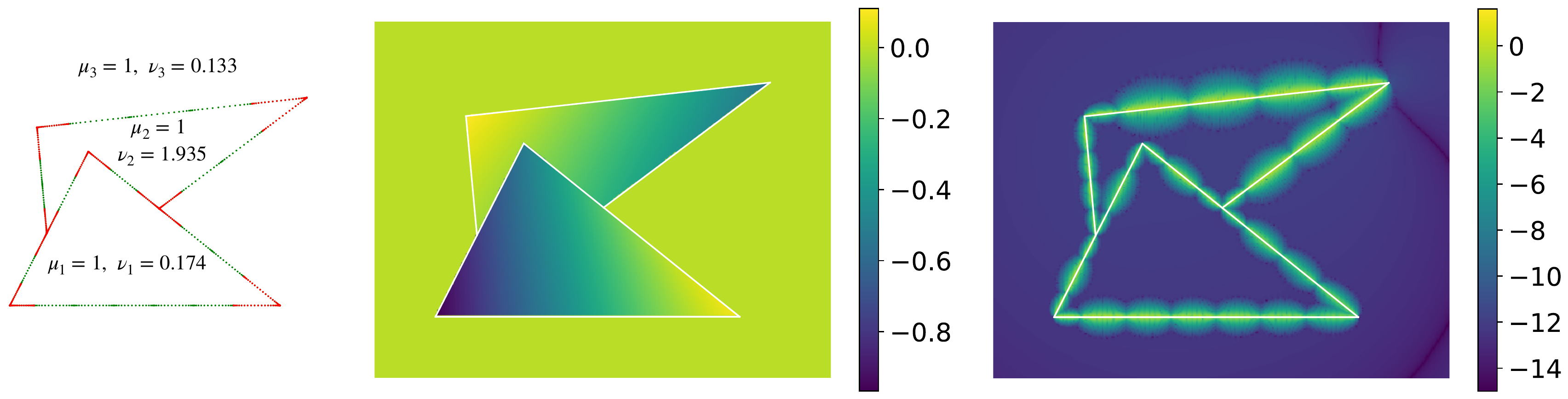}
\end{center}
\caption{(left): Discretization of geometry along with material parameters $\mu_{i},\nu_{i}$, 
the panels at corners/triple junctions are indicated in red, (center) exact solution $u_{j}$ in the domains, 
and (right) $\log_{10}$ of the absolute error in the solution. The geometry consists 7 vertices, 8 edges, 
3 regions, and is discretized with $768$ points. In order for the
solution of the linear system to converge to a residual
of $10^{-16}$, GMRES required $35$ iterations for~\cref{eq:inteq1}, and $48$ iterations for~\cref{eq:inteq2}}
\label{fig:2tri}
\end{figure}

\begin{figure}
\begin{center}
\includegraphics[width=\linewidth]{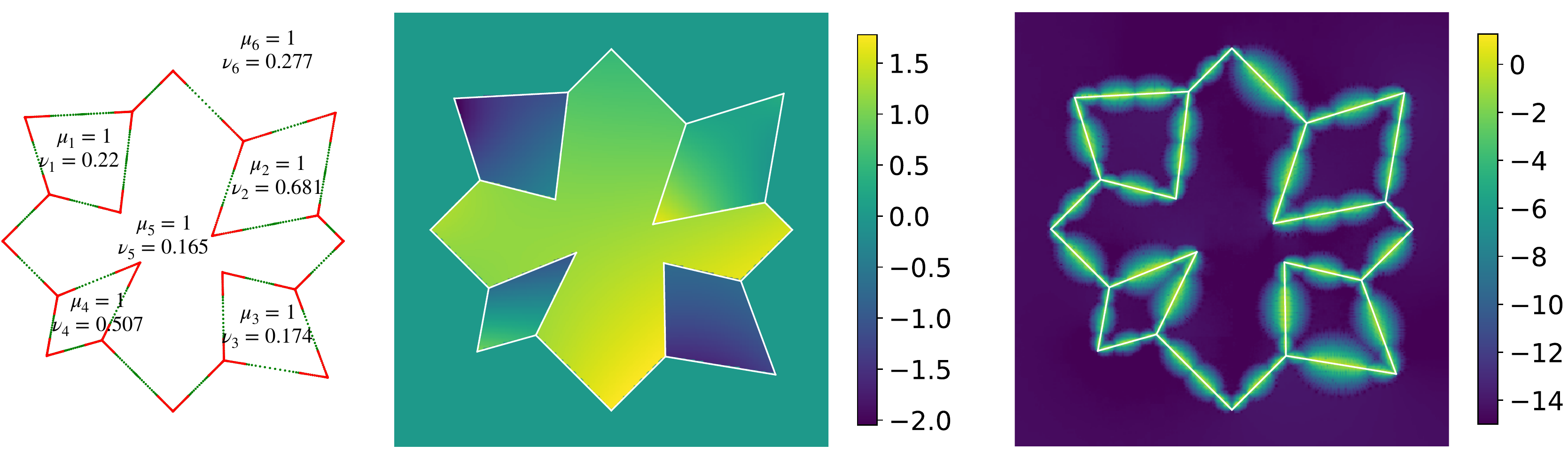}
\end{center}
\caption{(left): Discretization of geometry along with material parameters $\mu_{i},\nu_{i}$, 
the panels at corners/triple junctions are indicated in red, (center) exact solution $u_{j}$ in the domains, 
and (right) $\log_{10}$ of the absolute error in the solution. The geometry consists 20 vertices, 24 edges, 
5 regions, and is discretized with $1952$ points. In order 
for the solution of the linear system to converge to a residual
of $10^{-16}$, GMRES required $22$ iterations for~\cref{eq:inteq1}, and $28$ iterations for~\cref{eq:inteq2}}
\label{fig:4dia}
\end{figure}

\subsection{Condition number dependence on $\mu,\nu$ \label{subsec:cond1}}
In this section, we discuss the dependence of the condition number
of the discretized linear systems as a function of the material parameters
of the regions. 
Recall that the condition number of a linear system $A$,
which we denote by $\kappa(A)$, is the ratio 
of the largest singular value $\smax$ to the the smallest singular value $\smin$, 
i.e. $\kappa(A) = \smax/\smin$.
As discussed in section~\cref{sec:mainres}, for fixed angles the integral equation
and the analytical behavior of integral equations~\cref{eq:inteq1,eq:inteq2}
are solely a function of $d_{(1,2)},d_{(2,3)},d_{(3,1)}$ defined
in~\cref{eq:dijdef}.
Furthermore, $d_{(1,2)}$ can be expressed in terms
of $d_{(3,1)},d_{(2,3)}$ which are contained in the interval
$(-1,1)$.
As before, let $a=d_{(3,1)}$ and $b=d_{(2,3)}$.
Since the discrete linear system corresponding to~\cref{eq:inteq2},
is the adjoint of the linear system corresponding to~\cref{eq:inteq1},
it suffices to study the condition number 
for either linear system.

In~\cref{fig:matcond_nu}, we plot 
the condition number of the discretization of~\cref{eq:inteq1} 
as we vary $(a,b) \in (-1,1)^2$, by holding the values
of $\mu$ in each of the regions to be fixed.
In particular, we set $\mu_{1} = 0.37$,
$\mu_{2} = 0.81$, $\mu_{3} = 1$, and $\nu_{3} = 0.77$.
$\nu_{1}$, $\nu_{2}$ can then be defined in terms of
$(a,b)$ as 
\begin{equation}
\nu_{1} = \frac{\nu_{3}\mu_{1}}{\mu_{3}}\frac{1+a}{1-a} \, ,
\mu_{2} = \frac{\nu_{3}\mu_{2}}{\mu_{3}}\frac{1-b}{1+b} \, .
\end{equation}

We note that the problem is well-behaved for almost all values
of $(a,b)$ and becomes ill-conditioned as we approach the line
$b=-1$ and $a=1$.
This behavior is expected since the underlying physical
problem also has rank-deficiency along these limits 
since these values of the parameters correspond to interior Neumann problems
in regions $1$ and $2$ respectively. 

\begin{figure}
\begin{center}
\includegraphics[width=0.7\linewidth]{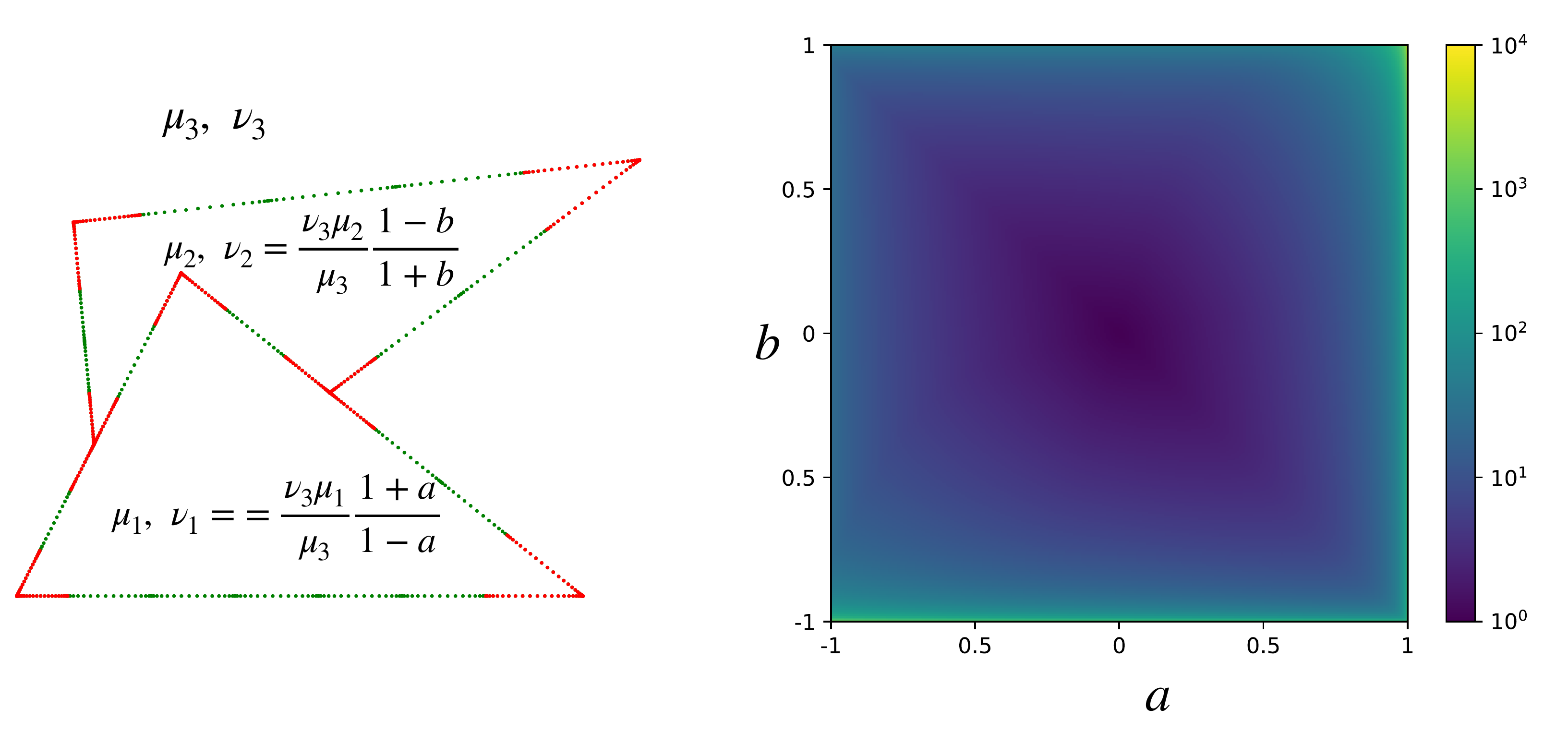}
\end{center}
\caption{(left) Discretization of geometry and material parameters $\mu,\nu$
as a function of $a,b$. 
Condition number as a function of $(a,b)$ with
$\mu_{1}=0.37$, $\mu_{2}=0.81$, $\mu_{3}=1$, and $\nu_{3}=0.77$.}
\label{fig:matcond_nu}
\end{figure}

\subsection{Condition number dependence on angles at the triple junction \label{subsec:cond2}}
In this section we discuss the the dependence of the condition number
of the discretized linear systems as a function of the angles at the triple
junction.
Let $\theta_{1},\theta_{2},\theta_{3}$, denote the angles
at the triple junction, then $\theta_{1} + \theta_{2} + \theta_{3} = 2\pi$.
The three angles at any triple junction can be parameterized by 
$\theta_{1},\theta_{2}$ in the simplex 
$\{(\theta_{1},\theta_{2}) : \theta_{1} > 0, \theta_{2} > 0 \,, \theta_{1}+\theta_{2}< 2\pi\}$.
Suppose that we split this simplex into $4$ regions as shown 
in~\cref{fig:cond_thet}. By symmetry it suffices to vary the angles $(\theta_{1},\theta_{2}) \in (0,\pi)^2$.

The physical problem as either of the angles approach $0$ or $2\pi$
becomes increasingly ill-conditioned due to close-to-touching interactions on 
the entire edge (not just near the corner). 
In order to avoid these issues and to automate geometry generation
as we vary the angles $\theta_{1},\theta_{2}$, we use two different
types of geometries for regions I and IV which are shown
in~\cref{fig:cond_thet}.
\begin{figure}
\begin{center}
\includegraphics[width=\linewidth]{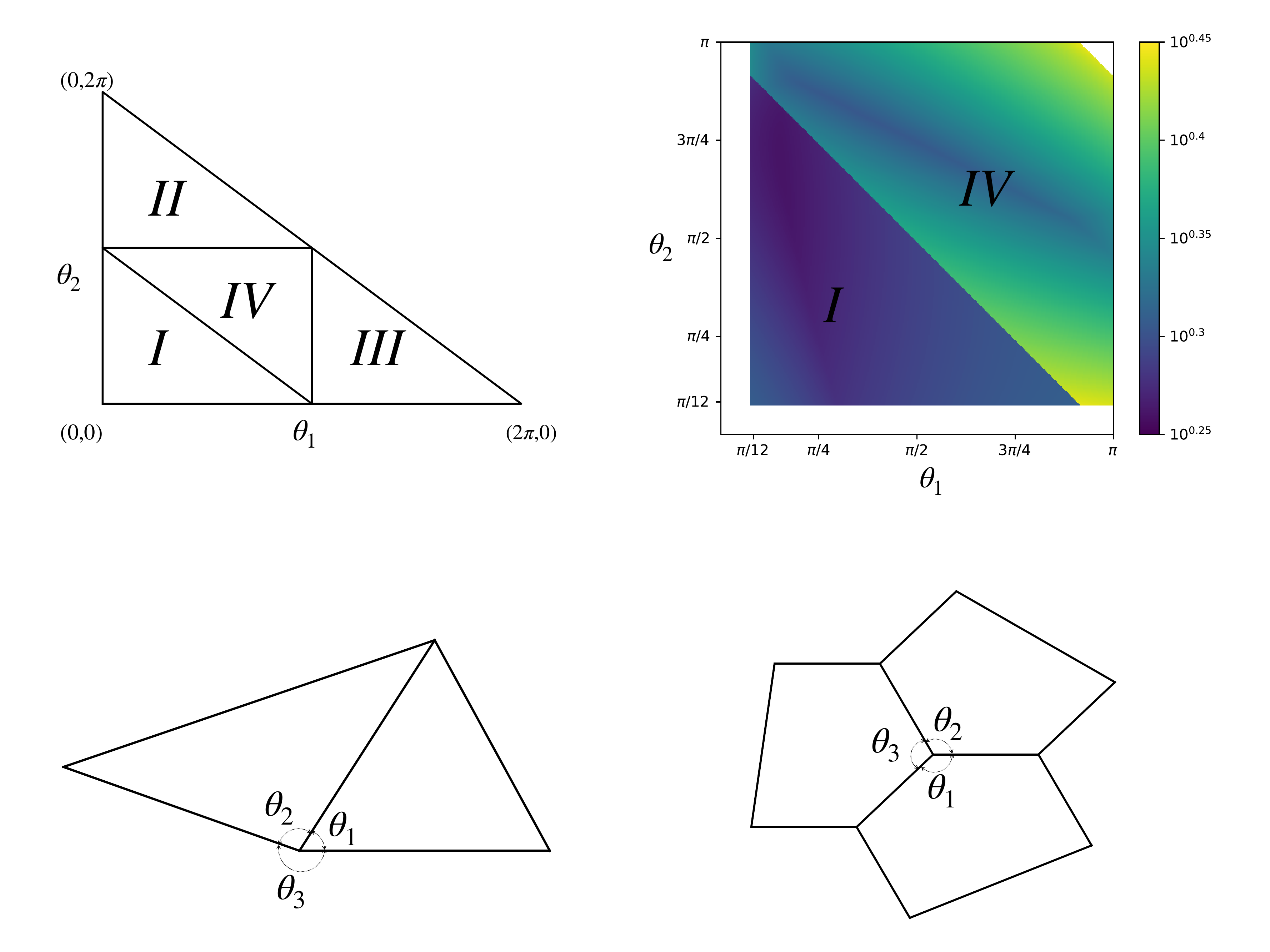}
\end{center}
\caption{(top left) Regions I-IV in $(\theta_{1},\theta_{2})$
simplex;
(top right) condition number of discretized linear system
corresponding to~\cref{eq:inteq1} as a function
of $(\theta_{1},\theta_{2})$;
(bottom left) sample domain for $(\theta_{1},\theta_{2})$ in 
region I; 
(bottom right) sample domain for $(\theta_{1},\theta_{2})$
in region IV.}
\label{fig:cond_thet}
\end{figure}

Resolving the close-to-touching interactions has
numerical consequences as well; due to the increased number
of quadrature nodes 
required as the angles tend to $0$ 
in the universal quadrature rules.
In order for the universal quadrature rules
to remain efficient, they are generated for the range 
$(\theta_{1},\theta_{2}) \in (\frac{\pi}{12}, 2\pi-\frac{\pi}{12})$.
Regions with narrower angles should be handled on case to case basis and 
region with careful
discretization of the boundary coupled with
special purpose quadrature rules 
which account for the specific singular
behavior of the solutions in the vicinity 
of triple junctions.
In~\cref{fig:cond_thet}, the top right missing corner corresponds to $\theta_{3} \in (0,\pi/12)$. 

Referring to~\cref{fig:cond_thet}, we observe that the condition number of the discrete linear
systems varies mildly as we vary the angles $\theta_{1},\theta_{2}$,
with a maximum condition number of $2.8$.
The discontinuity in the plot is explained by the different choice of
geometries for regions I, IV.

\subsection{Application - Polarization computation \label{subsec:app}}
In this section, we demonstrate the efficiency of our approach
for computing polarization tensors for a perturbed hexagonal
lattice with cavities.
The polarization computation corresponds to the following particular
setup of the triple junction problem, $\mu_{i} = 1$,
$f_{i} = 0$, 
$\nu_{i} = \eps_{i}$, where $\eps_{i}$ denotes
the permittivity of the medium, and 
$g_{1}(\bx) = (\eps_{\ell(i)} - \eps_{r(i)}) n_{1}(\bx)$ or $g_{2}(\bx) = (\eps_{\ell(i)} - \eps_{r(i)}) n_{2}(\bx)$, where $\bx=(x_{1},x_{2}) \in \Gamma_{i}$, $\bn(\bx) = (n_{1}(\bx),n_{2}(\bx))$, and $\eps_{\ell(i)}, \eps_{r(i)}$ are the conductivities of the regions on either side of the edge $\Gamma_{i}$.
If $u_{1}$ is the solution corresponding to $g_{1}$
and $u_{2}$ is the solution corresponding to $g_{2}$, 
then the polarization tensor $P$ is the $2\times 2$ matrix given by
\begin{equation}
P = \begin{bmatrix}
\int_{\Gamma} x_{1} \cdot \frac{\partial u_{1}}{\partial n} \, ds &
\int_{\Gamma} x_{2} \cdot \frac{\partial u_{1}}{\partial n} \, ds \\
\int_{\Gamma} x_{1} \cdot \frac{\partial u_{2}}{\partial n} \, ds &
\int_{\Gamma} x_{2} \cdot \frac{\partial u_{2}}{\partial n} \, ds 
\end{bmatrix} \, .
\end{equation}

Note that in this particular setup, we only need to solve the problem
corresponding to the operator $\Kneu$, as the solution $\sigma$
for $\Kdir[\sigma] = 0$ is $\sigma=0$.
Let $\rho_{1},\rho_{2}$ denote the solutions of~\cref{eq:inteq2}
corresponding to boundary data $g_{1}$ and $g_{2}$
respectively.
Using properties of the single layer potential, the integrals 
of the polarization tensor can be expressed in terms of $\rho$
as 
\begin{equation}
P = \begin{bmatrix}
\int_{\Gamma} x_{1} \cdot \rho_{1} \, ds &
\int_{\Gamma} x_{2} \cdot \rho_{1} \, ds \\
\int_{\Gamma} x_{1} \cdot \rho_{2} \, ds &
\int_{\Gamma} x_{2} \cdot \rho_{2} \, ds 
\end{bmatrix} \, .
\end{equation}

We compare the efficiency of our approach to RCIP which to the best of our knowledge
is the state of the art method for such problems. The geometry is generated using 
a regular hexagonal lattice inside the unit square 
whose vertices are perturbed in a random direction by a tenth of the side length, and the
permittivity $\eps$ is region $i$ is given by $10^{c_{i}}$ where $c_{i}$ is a uniform
random number between $[-1,1]$. The choice of parameters for the problem setup
is identical to the setup in section 11 in~\cite{helsing2008corner}. 

We discretize the geometry with $3$ panels
on each edge of roughly equal size, and the reference solution is computed using
$5$ panels on each edge. The geometry contains $10688$ vertices, $15855$ edges, 
and $5189$ regions. There are $1395240$ degrees of freedom for the coarse discretization 
(approximately $88$ degrees of freedom per edge)
and $1902600$ degrees of freedom for the reference solution. 
These discretizations required $131$ iterations for GMRES to converge to a relative residual
of $10^{-16}$, and the absolute error in the polarization tensor when compared to
the reference solution is $5.1 \times 10^{-12}$. 
In comparison, RCIP using approximately $71$ degrees of freedom per edge in a hexagonal
lattice with $5293$ inclusions obtained 
an accuracy of $2 \times 10^{-14}$ in computing the $2,2$ entry of the polarization matrix and required
$105$ GMRES iterations to converge.

The polarization tensor for this configuration, correct to 13 significant
digits, is given by
\begin{equation}
P = \begin{bmatrix}
-0.038291586646 & -0.004056508957 \\
-0.004056508957 & 0.045585776453 
\end{bmatrix} \, .
\end{equation}

\begin{remark}
We note that the performance of our approach is close to current state of the art methods such as RCIP~\cite{helsing2013solving}. In our examples, further improvements in speed can be achieved using additional compression techniques to reduce the degrees of freedom in the resulting linear system~\cite{bremer2015high,greengard2009fast}.
\end{remark}

\begin{figure}
\begin{center}
\includegraphics[width=\linewidth]{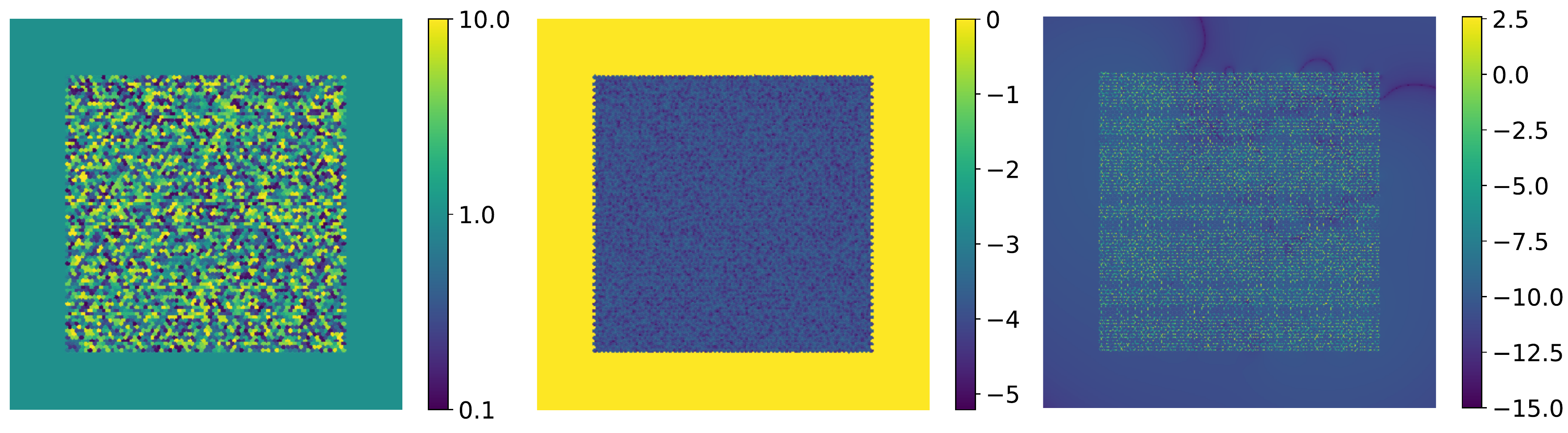}
\end{center}
\caption{(left): Material parameters $\nu_{i}$ for each of the regions, 
(center) exact solution $u_{j}$ in the domains, 
and (right) $\log_{10}$ of the absolute error in the solution. 
The geometry consists $10688$ vertices, $15855$ edges, 
$5189$ regions, and is discretized with $1395240$ points. 
In order for the solution of the linear system to converge to a residual
of $10^{-16}$, GMRES required $138$ iterations for~\cref{eq:inteq1}, and $130$ 
iterations for~\cref{eq:inteq2} in the accuracy tests, 
and $131$ and $130$ iterations (for $g_{1}$, and $g_{2}$ respectively) 
for~\cref{eq:inteq2} in computing the polarization
tensors.}
\label{fig:tetris}
\end{figure}


\section{Concluding remarks and future work \label{sec:conc}}
In this paper we analyze the systems of boundary integral equations which arise 
when solving the Laplace transmission problem in composite media consisting of
regions with polygonal boundaries. Our discussion is focused on the particular 
case of composite media with triple junctions (points at which three distinct
 media meet) though our analysis extends to higher-order junctions in a natural way.  

We show that under some restrictions the solutions to the boundary integral equations 
corresponding to a triple junction is well-approximated by a linear combination 
of powers $t^{\beta_j}$ where $t$ denotes the distance from the corner along the edge,
and the $\beta_j,$ $j=1,2,\dots$ is a countable collection of real numbers obtained by
solving a certain equation depending only on the material properties of the media 
and the angles at which the interfaces meet.

In addition to the theoretical interest of the result, our analysis also enables 
an easy construction of near-optimal discretizations for triple junctions.
In particular, RCIP which is the leading method for solving electrostatic problems
on multiple junction interfaces, requires approximately $71$ discretization nodes per edge
to compute solutions to near machine precision accuracy, where as our proposed discretization
achieves an accuracy of $5 \times 10^{-12}$ using roughly $88$ discretization nodes per edge.
Finally, we illustrate the properties of this discretization with a number of 
numerical examples.

The results of this paper admit a number of natural extensions and generalizations. Firstly,
the analysis outlined in this paper extends almost immediately to junctions involving greater
numbers of media. However, the construction of an efficient Nystr\"{o}m discretization
of higher-order junctions requires special care since the 
solutions to corresponding integral equations are not $\mathbb{L}^{2}$ functions on the boundary.
(in fact the solutions are known to be $\mathbb{L}^{1}$ functions on the boundary~\cite{helsing2011effective}).
Secondly, with a small modification a 
similar analysis should be possible for boundary integral equations arising from triple 
junction problems for other partial differential equations such as the Helmholtz equation,
Maxwell's equations, and the biharmonic equation. This line of inquiry is being vigorously 
pursued and will be reported at a later date.
  
Finally, a similar approach will also work for generating discretizations of triple junctions in 
three dimensions. This is particularly valuable since geometric singularities in 
three-dimensions can often result in prohibitively large linear systems. Accurate discretization 
with few degrees of freedom would greatly improve the size and complexity of systems which 
could be simulated.


\section{Acknowledgments}
The authors would like to thank Alex Barnett, Charles Epstein, 
Leslie Greengard, Vladimir Rokhlin, and Kirill Serkh 
for many useful discussions. 


\appendix
\newpage
\section{Analysis of $\Kdir$ \label{sec:appendix-dir}}

First we present the proof of~\cref{thm_fwd}. 
In order to do so, we require the following technical lemma
which describes the double layer potential defined on 
a straight line segment with density $s^{\beta}$ 
at an arbitrary point near the boundary.
Here $s$ is the distance along the segment.
\begin{lemma}
\label{lem:dledge}
Suppose that $\Gamma$ is an edge of unit length oriented
along an angle $\theta$, parameterized by 
$s(\cost,\sint)$, $0<s<1$. 
Suppose that $\bx= t(\cos{(\theta+\theta_{0})}, \sin{(\theta+\theta_{0})})$
(see~\cref{fig_illus_ints})
where $0<t<1$, and $\bx \not \in \Gamma$.
Suppose that $\sigma(s) = s^{\beta}$ for $0<s<1$, where
$\beta \geq 0$.
If $\beta$ is not an integer, then
\begin{equation}
\cD_{\Gamma}[\sigma](\bx) = \frac{\sin{(\beta(\pi-\theta_{0}))}}{2\sinb}t^{\beta}
+\frac{1}{2\pi}\sum_{k=1}^{\infty} \frac{\sin{(k\theta_{0})}}{\beta-k} t^{k} \, .
\end{equation}
If $\beta=m$ is an integer, then
\begin{equation}
\cD_{\Gamma}[\sigma](\bx) = \frac{(\pi-\theta_{0})\cos{(m \theta_{0})}}{2\pi}t^{m}
-\frac{\sin{(m\theta_{0})}}{2\pi} t^{m} \log{(t)}
+\frac{1}{2\pi}\sum_{\substack{k=1\\k \neq m}}^{\infty} 
\frac{\sin{(k\theta_{0})}}{m-k} t^{k} \, .
\end{equation}
\end{lemma}

\begin{figure}
\begin{center}
\includegraphics[width=0.3\linewidth]{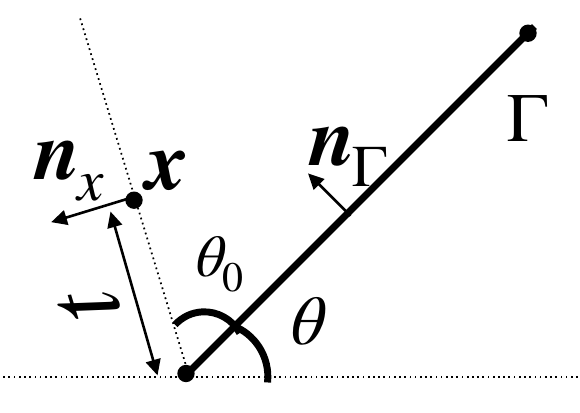}
\caption{Illustrative figure for geometry in~\cref{lem:dledge}}
\label{fig_illus_ints}
\end{center}
\end{figure}

In the following lemma, 
we compute the potential $\Kdir[\bv t^{\beta}]$, 
in the vicinity of triple junction with angles $\theta_{1},\theta_{2}, 
\theta_{3}$, material parameters $\bd = 
(d_{(1,2)},d_{(2,3)},d_{(3,1)})$, 
where $\bv \in \mathbb{R}^{3}$ and $\beta$ are constants
(see~\cref{fig:trip_geom}).

\begin{lemma}
\label{lem:pot_fwd}
Consider the geometry setup of the single vertex problem
presented in~\cref{sec:mainres}.
For a constant vector $\bv \in \mathbb{R}^{3}$, suppose 
that the density on the edges is of the form
\begin{equation}
\sigma =
\begin{bmatrix}\sigma_{1,2} \\ \sigma_{2,3} \\ \sigma_{3,1}
\end{bmatrix} = 
\bv t^{\beta}
\end{equation}
If $\beta$ is not an integer, then 
\begin{equation}
\Kdir[\sigma] = 
-\frac{1}{2\sinb} \Adir(d_{3,1},d_{1,2},\beta) \bv
t^{\beta}
+ \sum_{k=1}^{\infty} \frac{1}{\beta-k}\bC(\bd,k) \bv 
t^{k} \, ,
\end{equation}
where
$\Adir$ is defined in~\cref{eq:defAdir} 
and 
\begin{equation}
\label{eq:defcsmooth}
\bC(\bd,k)
=
\frac{1}{2 \pi}
\begin{bmatrix}
0 & -d_{(1,2)}\sin{(k\theta_{2})} & d_{(1,2)}\sin{(k\theta_{1})} \\
d_{(2,3)}\sin{(k\theta_{2})} & 0 & -d_{(2,3)}\sin{(k\theta_{3})} \\
-d_{(3,1)}\sin{(k\theta_{1})} & d_{(3,1)}\sin{(k\theta_{3})} & 0
\end{bmatrix} \, .
\end{equation}
If $\beta = m$ is an integer, then
\begin{equation}
\begin{aligned}
\Kdir[\sigma] &= 
-\frac{(-1)^{m}}{2\pi}\Adir(d_{3,1},d_{1,2},m) \bv
t^{m} \log{(t)}
+ \sum_{\substack{k=1\\k\neq m}}^{\infty} \frac{1}{m-k}\bC(\bd,k)\bv t^{k} + 
\bCdiag(\bd,m) \bv t^{m}\, ,
\end{aligned}
\end{equation}
where
{\small
\begin{equation}
\label{eq:defcsmoothint}
\bCdiag(\bd,m) =
-\frac{1}{2\pi}
\begin{bmatrix}
\pi & d_{(1,2)}(\pi-\theta_{2}) \cos{(m\theta_{2})} & -d_{(1,2)}(\pi-\theta_{1}) 
\cos{(m \theta_{1})} \\
-d_{(2,3)}(\pi-\theta_{2}) \cos{(m\theta_{2})} &  \pi &
d_{(2,3)}(\pi-\theta_{3})\cos{(m\theta_{3})} \\
d_{(3,1)}(\pi-\theta_{1})\cos{(m\theta_{1})} & -d_{(3,1)}(\pi-\theta_{3})
\cos{(m\theta_{3})} & \pi
\end{bmatrix} \, .
\end{equation}
}
\end{lemma}
\begin{proof}
The result follows from repeated application of~\cref{lem:dledge}
for computing $\cD_{(l,m):(i,j)} \sigma_{(i,j)}$.
\end{proof}
The proof of~\cref{thm_fwd} then follows immediately 
from~\cref{lem:pot_fwd}.

We now turn our attention to the 
proof of~\cref{lem:betavs_existence_dir}, which
provides a construction of $\beta,\bv$ satisfying the
conditions of~\cref{thm_fwd}.
In order to do that, we first observe that if one of $a,b$, or $c$ is $0$,
then the expression of $\det{\Adir}$ simplifies significantly, 
and there exists an explicit construction of $\beta$ satisfying
$\det{\Adir}(a,b,\beta) = 0$.
Recall that we use interchangeably use the following variables
for the material properties $(a,b,c) = (d_{3,1},d_{1,2},d_{2,3})$.
Having established the existence of analytic $\beta,\bv$ on a 
1-D manifold which is a subset $(a,b) \in (-1,1)^2$, 
we now analytically continue these values of $\beta,\bv$
to carve out the open region $S$ on which $\beta,\bv$ can be
analytically extended. 
This proof is discussed in~\cref{subsec:beta_v_existence}.

\subsection{Existence of $\beta,\bv$ satisfying~\cref{thm_fwd}}
\label{subsec:beta_v_existence}
The determinant of the matrix $\Adir(a,b,\beta)$ 
is given by
\begin{equation}
\det{\Adir(a,b,\beta)} = \spt{\beta} \alpha\left(a,b,c; 
\beta \right) \, , 
\label{eq:detAdir}
\end{equation}
where
$c = -(a+b)/(1+ab)$, and 
\begin{equation}
\alpha(a,b,c;\beta) = \sptsq{\beta} + bc \sinsq{\beta(\pi-\theta_{2})} +
ac \sinsq{\beta(\pi-\theta_{3})} + ab \sinsq{\beta(\pi-\theta_{1})} \, .
\label{eq:defalpha}
\end{equation}
Given the formula above, 
for  all $(a,b) \in (-1,1)^2$
when $\beta=m\geq 0$ is an integer, $\det{\Adir(a,b,\beta)} = 0$.
When $m\neq 0$, the matrix $\Adir$ has rank-2, since the matrix is
similar to an anti-symmetric matrix and is not identically zero.
The null vector $\bv$ of $\Adir(a,b,m)$ is given by
$\bv_{m}= [\sin{(m\theta_{3})},\sin{(m\theta_{1})},\sin{(m\theta_{2})}]^{T}$,
i.e., the pair $(m,\bv_{m})$ always satisfies~\cref{eq:defAdir}.
When $\beta=0$, $\Adir(a,b,\beta) = 0$ 
and hence for any $\bv \in \bR^{3}$, the pair $\beta,\bv$
satisfies~\cref{eq:defAdir}.
Based on this observation we set
\begin{equation}
\begin{aligned}
\beta_{m,0} = m\, , &\quad \bv_{m,0} = 
[ 
\sin{(m\theta_{3})},
\sin{(m\theta_{1})},
\sin{(m\theta_{2})}]^{T} \, , \quad S_{m,0} = (-1,1)^2\\
\beta_{0,0} = 0 \, , &\quad \bv_{0,0} = [1,0,0]^{T} \, , \quad S_{0,0} = (-1,1)^2 \\
\beta_{1,0} = 0 \, ,  &\quad \bv_{1,0} = [0,1,0]^{T} \, , \quad S_{0,1} = (-1,1)^2\\
\beta_{2,0} = 0 \, , &\quad \bv_{2,0} = [0,0,1]^{T} \, , \quad S_{0,2} = (-1,1)^2.
\end{aligned}
\end{equation}

We now turn our attention to constructing the remaining 
$\beta_{i,j}$, 
the corresponding vectors $\bv_{i,j}$,
and their regions of analyticity
$S_{i,j}$, $i=1,2,\ldots $, $j=1,2$. 
From~\cref{eq:detAdir}, 
the remaining values of $\beta_{i,j}$ as a function of the
material parameters $(a,b)$ are defined implicitly
via the roots of the equation $\alpha(a,b,c(a,b);\beta_{i,j}(a,b)) = 0$,
where $c=-(a+b)/(1+ab)$ and $\alpha$ is defined in~\cref{eq:defalpha}.

It turns out that the implicit solutions $\beta(a,b)$ of
$\alpha(a,b,c(a,b);\beta(a,b)) = 0$,
are known when $a=0$, $b=0$, or $c=0$. 
This gives us an initial value for defining $\beta_{i,j}$  
in order to apply the implicit function theorem, and 
extend it to a region containing the segments $a=0$, $b=0$, or $c=0$.
Given this strategy, 
let $R_1,\dots,R_6 \subset (-1,1)\times(-1,1)$ be defined as follows (see~\cref{fig:power_analytic})
\begin{align}
R_1 &= \{(x,0) : x>0 \},\\
R_2 &= \{(-x,0) : x>0 \},\\
R_3 &= \{(0,x) : x>0 \},\\
R_4 &= \{(0,-x) : x>0 \},\\
R_5 &= \{(-x,x) : x>0 \},\\
R_6 &= \{(x,-x) : x>0 \}.
\end{align}
\begin{figure}[!h]
\centering
\includegraphics[width=0.4\linewidth]{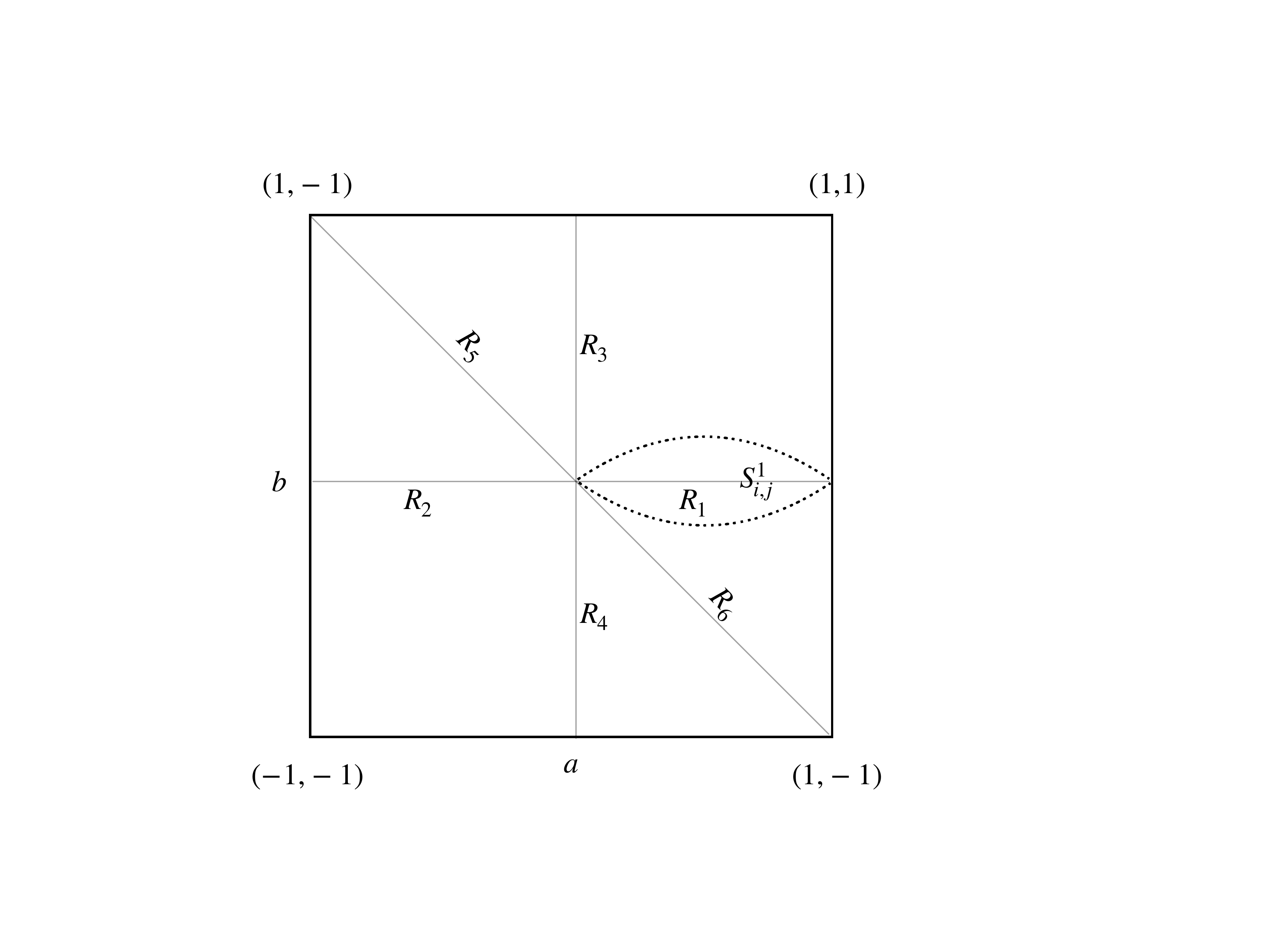}
\caption{Illustration of the edge segments $R_{i}$, $i=1,2,\ldots 6$, and
a typical region of analyticity of $\beta_{i,j}$ denoted by 
$S_{i,j}^{1}$.}
\label{fig:power_analytic}
\end{figure}

In the following, we will consider only the segment $R_1$; 
and construct an open region $S_{i,j}^{1} \subset (-1,1)^2$ which contains
$R_{1}$ on which we define a family of functions 
$\beta_{i,j}(a,b): S^{1}_{i,j} \to \mathbb{R}$,
$j=1,2$, which satisfy the conditions of~\cref{lem:betavs_existence_dir}.
Analogous results hold for the open sets containing the
remaining segments $R_{2}, R_{3}, \ldots R_{6}$ with almost identical proofs.
The region of analyticity for $\beta_{i,j}$ is then given by
$S_{i,j} = \cup_{k=1}^{6} S_{i,j}^{k}$.

\begin{definition}
For $(a,0) \in R_1$ and $i=1,2\dots$ 
let $\beta_{i,1}(a,0)$ be the solution to the equation
\begin{align}
\sin{(\pi \beta_{i,1})} =  -a \sin{(\beta_{i,1}(\pi-\theta_3))}
\end{align}
such that 
\begin{align}
\lim_{a \to 0} \beta_{i,1}(a,0) = i.
\end{align}
Similarly, for $i=1,2,\dots$ let $\beta_{i,2}(a,0)$ 
be the solution to the equation
\begin{align}
\sin{(\pi \beta_{i,2})} =  a \sin{(\beta_{i,2}(\pi-\theta_3))}
\end{align}
such that 
\begin{align}
\lim_{a \to 0} \beta_{i,2}(a,0) = i.
\end{align}
\end{definition}

The existence of $\beta_{i,j}$ for $i=1,2,\dots$ and $j=1,2$ 
satisfying these conditions is guaranteed by the following 
~\cref{lem_transm}, proved in \cite{hoskins2018numerical}.

\begin{lemma}\label{lem_transm}
Suppose that $\delta \in \mathbb{R}$, $0<|\delta| <1$ and 
$\theta \in (0,2\pi)$ and $\theta/\pi$ is irrational. 
Consider the equations
$$\sin(\pi z) = \pm\delta \sin(z(\pi-\theta)) \, .$$
Then there exist an countable collection of 
functions $z_i^\pm(\delta),$ $i=1,2,\dots$ such that
\begin{enumerate}
\item $\sin^2(\pi z_i^\pm(\delta)) = \delta^2 \sin^2(z_i^\pm(\delta)(\pi-\theta))$ for all $\delta\in [0,1],$ and $i=1,2,\dots$
\item the functions $z_i^\pm$ are analytic in $(0,1)$,
\item $\lim_{\delta\to 0} z_i^\pm(\delta) = i,$
\item $z_i^+(\delta)>i$ and $z_i^-(\delta)<i$ for all $\delta \in (0,1)\,.$
\end{enumerate}
\end{lemma}

The following lemma extends the domain of definition of the 
functions $\beta_{i,j}$, $j=1,2$, 
to some open subset $S_{i,j}^{1}$ containing $R_1$.

\begin{lemma}
Suppose $\theta_1,$ $\theta_2$ and $\theta_3$ are positive numbers
summing to $2\pi$, and $\theta_{1}/\pi$, $\theta_{2}/\pi$, and $\theta_{3}/\pi$ 
are irrational numbers. 
Suppose that $\beta_{i,j}$ are defined as above for $i=1,2,\dots$ and $j=1,2$. 
For $a\in (0,1)$, 
the function $\beta_{i,j}$ satisfies
\begin{align}
\alpha(a,0,-a;\beta_{i,j})= 0.
\end{align}
Moreover, there exists a unique extension of 
$\beta_{i,j}$ to an analytic function of $(a,b)$ 
on an open neighborhood $R_{1}\subset S_{i,j}^{1} \subset (-1,1)^2$ 
which satisfies
\begin{align}
\alpha(a,b,c(a,b);\beta_{i,j}) = 0 \, .
\end{align}
\end{lemma}
\begin{proof}
We begin by observing that for $j=1,2,$ $\beta_{i,j}$ satisfies
\begin{equation}\label{eqn:lem_an_eq1}
\alpha(a,0,-a;\beta_{i,j}) = 
-a^2\sin^2(\beta_{i,j} (\pi-\theta_3)) + \sin^2(\pi \beta_{i,j})=0 \, .
\end{equation}
$$\frac{\partial \alpha}{\partial \beta}(a,0,-a;
\beta_{i,j}) = 
2 \left( -(\pi-\theta_3)a^2\sin(\beta_{i,j} (\pi-\theta_3)) \cos(\beta_{i,j} 
(\pi-\theta_3)) + \pi\sin(\pi \beta_{i,j}) \cos(\pi \beta_{i,j})\right).$$
Upon multiplication by 
$$g(a;\beta) = (\pi-\theta_3)a^2\sin(\beta (\pi-\theta_3)) \cos(\beta 
(\pi-\theta_3)) + \pi\sin(\pi \beta) \cos(\pi \beta)$$
and using~\cref{eqn:lem_an_eq1} we get 
$$g(a;\beta_{i,j})
\frac{\partial \alpha}{\partial \beta}(a,0,-a;\beta_{i,j})= 
-2\sin^2(\pi\beta_{i,j}) 
\left(\pi^2-a^2(\pi-\theta_3)^2  -(\pi^2-(\pi-\theta_3)^2)\sin^2(\pi \beta_{i,j})\right) \, ,$$
which does not vanish for all $a>0$. 
Thus by the implicit function theorem,
there exists an analytic extension of $\beta_{i,j}$ 
to a neighborhood $(a,b)\in R_{1} \subset S_{i,j}^{1} \subset (-1,1)^2$ 
which satisfies
$\alpha(a,b,c(a,b);\beta_{i,j}) = 0$.
\end{proof}

The following theorem establishes 
the analyticity of the null vectors of $\Adir(a,b;\beta)$ 
in a neighborhood of $R_1$ when $\beta=\beta_{i,j}.$
\begin{theorem}
For each $j=1,2$, and $i=1,2,\dots$, the matrix
$\Adir(a,b,\beta_{i,j})$ defined in~\cref{eq:defAdir}
has a null-vector $\bv_{i,j}$ whose entries are analytic
functions of $(a,b)$ on $S_{i,j}^{1}$.
\end{theorem}
\begin{proof}
Since $\beta_{i,j}$ is such that the matrix $\Adir(a,b,\beta_{i,j})$
is singular, it has a null vector $\bv_{i,j}$.
Moreover, as long as $(a,b) \neq (0,0)$ and $\beta_{i,j}$ 
is not an integer, the matrix $\Adir$ has rank at least $2$.
Thus $0$ is an eigenvalue of $\Adir(a,b,\beta_{i,j}(a,b))$ with
multiplicity $1$ for all $(a,b) \in S_{i,j}^{1}$.
Since the entries of the matrix $\Adir$ are analytic functions
of $(a,b)$, we conclude that the entries of $\bv_{i,j}$ are analytic
on $S_{i}^{1}$.
\end{proof}

Finally, each $S_{i,j}^{k}$ is an open subset containing the segments
$R^{k}$, $k=1,2,\ldots 6$.
Then $S_{i,j} = \cup_{k=1}^{6} S_{i,j}^{k}$ is an open subset of 
$(-1,1)^2$ containing $\cup_{k=1}^{6} R_{k}$. 
Thus, for any finite $N$, $|\cap_{i=0}^{N}\cap_{j=0}^{2} S_{i,j}|>0$.

\subsection{Completeness of density basis}
\label{subsec:smooth_onto}
Recall that for any $\beta,\bv$ which satisfy
the conditions of~\cref{thm_fwd}, and $\sigma = \bv t^{\beta}$,
the potential $\Kdir[\sigma]$ corresponding to any of these densities
is an analytic function. 
In order to show that, the potential corresponding to a particular
collection of $\beta,\bv$ span all polynomials of a fixed degree 
on all the three edges meeting at the triple junction, we 
explicitly write down the linear map 
from the coefficients of the density in the $\bv t^{\beta}$ basis
to the coefficients of Taylor series of the potentials on each of 
the edges using~\cref{lem:pot_fwd}.
We then observe that this mapping is invertible along 
the line segments corresponding to $a=0$, $b=0$, or
$c=0$, and since the mapping is an analytic function
of the parameters $(a,b)$, it must also be invertible
in an open region containing the segments $a=0$, $b=0$ or $c=0$.
This part of the proof is discussed in~\cref{subsec:smooth_onto}.

For any integer $N>0$, let $\hat{S}_{N}$, 
denote the common region of analyticity of $\beta_{i,j},\bv_{i,j}$,
$j=0,1,2$, $i=0,1,2\ldots N$, i.e.
$S_{N} = \cup_{k=1}^{6} S^{k}_{N}$, where
$S_{N}^{k} = \cap_{i=0}^{N} \cap_{j=0}^{2} S_{i,j}^{k}$. 
By construction, $R_{j} \subset S^{j}_{N}$ for all $N$.
We now prove the result~\cref{thm_inv} 
in one of the components of $S_{N}$, say $S_{N}^{1}$.
The proof for the other components follows in a similar manner. 

Let $\bp_{i} = [p_{i,0},p_{i,1},p_{i,2}]^{T}$, and 
suppose that 
\begin{equation}
\sigma(t) = \sum_{i=0}^{N} \sum_{j=0}^{2} 
p_{i,j} \bv_{i,j} |t|^{\beta_{i,j}} \, .
\end{equation}
Then, using~\cref{lem:pot_fwd}, 
since $\beta_{i,j},\bv_{i,j}$ are such that
$\Adir(a,b,\beta_{i,j}) \cdot \bv_{i,j} = 0$, 
the potential corresponding to this density on the
boundary $(\Gamma_{(1,2)},\Gamma_{(2,3)},\Gamma_{(3,1)})$
is given by
\begin{equation}
\begin{bmatrix}
u_{(1,2)}(t) \\
u_{(2,3)}(t) \\
u_{(3,1)}(t)
\end{bmatrix}= \sum_{i=0}^{N} \left(\sum_{j=0}^{N} B_{i,j} \cdot \bp_{j} \right) |t|^{i} 
+ O(|t|^{N+1})\, ,
\end{equation}
where
$B_{i,j}$ are the $3 \times 3$ matrices given by
\begin{equation}
B_{i,j} = 
\begin{cases}
\left[
\begin{array}{c;{2pt/2pt}c;{2pt/2pt}c}
\frac{1}{\beta_{j,0}-i}\bC(\bd,i) \bv_{j,0} &
\frac{1}{\beta_{j,1}-i}\bC(\bd,i) \bv_{j,1} &
\frac{1}{\beta_{j,2}-i}\bC(\bd,i) \bv_{j,2}
\end{array}
\right] & \quad \text{if } i\neq j \vspace*{1.5ex}\\
\left[
\begin{array}{c;{2pt/2pt}c;{2pt/2pt}c}
\bCdiag(\bd,i) \bv_{j,0} &
\frac{1}{\beta_{j,1}-i}\bC(\bd,i) \bv_{j,1} &
\frac{1}{\beta_{j,2}-i}\bC(\bd,i) \bv_{j,2}
\end{array}
\right] & \quad \text{if } i=j\neq 0 \vspace*{1.5ex}\\ 
\left[
\begin{array}{c;{2pt/2pt}c;{2pt/2pt}c}
\bCdiag(\bd,i) \bv_{j,0} &
\bCdiag(\bd,i) \bv_{j,1} &
\bCdiag(\bd,i) \bv_{j,2}
\end{array}
\right] & \quad \text{if } i=j=0 \\ 
\end{cases} \, .
\end{equation}
Let $\bB$ denote the $3(N+1)\times 3(N+1)$ matrix
whose $3\times 3$ blocks are given by $B_{i,j}$, 
$i,j=0,1,2,\ldots N$.

Recall that on $R_{1}\subset S_{N}^{1}$, $b=0$, 
$\beta_{i,0} = i$,
$\beta_{i,1}$,
satisfies $\spt{\beta_{i,1}} = -a \sin{(\beta_{i,1}(\pi-\theta_{3}))}$,
$\beta_{i,2}$ satisfies $\spt{\beta_{i,2}} = a \sin{(\beta_{i,2}(\pi-\theta_{3}))}$\,,
$i=0,1,2\ldots$, 
and the corresponding vectors $\bv_{i,j}$, $i=0,1,2\ldots$, 
$j=0,1,2$, are given by 
\begin{equation}
\bv_{i,0} = \frac{1}{\eta_{i}}\begin{bmatrix} \sin{(i\theta_{3})} \\
\sin{(i\theta_{1})} \\ \sin{(i\theta_{2})} 
\end{bmatrix} \, ,
\quad
\bv_{i,1} = \frac{1}{\sqrt{2}}\begin{bmatrix} 0 \\ 1 \\ 1 
\end{bmatrix} \, ,
\quad
\bv_{i,2} = \frac{1}{\sqrt{2}}\begin{bmatrix} 0 \\ 1 \\ -1 
\end{bmatrix} \, ,
\end{equation}
where
\begin{equation}
\eta_{i} = \sqrt{\sinsq{i\theta_{1}} + \sinsq{i\theta_{2}} + \sinsq{i\theta_{3}}}
\end{equation}
Furthermore, the matrices $\bC$ and $\bCdiag$ defined in~\cref{eq:defcsmooth},
and~\cref{eq:defcsmoothint} respectively, also simplify to
\begin{equation}
\bC = \frac{a}{2\pi} 
\begin{bmatrix}
0 & 0 & 0 \\
-\sin{(m\theta_{2})} & 0 & -\sin{(m\theta_{3})} \\
-\sin{(m\theta_{1})} & \sin{(m\theta_{3})} & 0
\end{bmatrix} \, ,
\end{equation}
and 
\begin{equation}
\bCdiag = -\frac{1}{2\pi} 
\begin{bmatrix}
\pi & 0 & 0 \\
a(\pi-\theta_{2}) \cos{(m\theta_{2})} &\pi  & -a(\pi-\theta_{3}) \cos{(m\theta_{3})} \\
a(\pi-\theta_{1}) \cos{(m\theta_{1})} & -a(\pi-\theta_{3})\cos{(m\theta_{3})} & \pi
\end{bmatrix}
\end{equation}

Let $u_{(1,2),i}, u_{(2,3),i}, u_{(3,1),i}$ denote
the coefficient of $|t|^{i}$ in the Taylor expansions of
$u_{(1,2)},u_{(2,3)},u_{(3,1)}$ respectively.
Let $P$ denote the permutation matrix whose action is given by
\begin{equation}
P 
\begin{bmatrix}
p_{0,0} \\ p_{0,1} \\ p_{0,2} \\ \hdashline[0.5pt/1pt]
p_{1,0} \\ p_{1,1} \\ p_{1,2} \\ \hdashline[0.5pt/1pt]
\vdots \\ \hdashline[0.5pt/1pt]
\vdots \\ \hdashline[0.5pt/1pt]
\vdots \\ \hdashline[0.5pt/1pt]
p_{N,0} \\ p_{N,1} \\ p_{N,2}
\end{bmatrix}
= 
\begin{bmatrix}
p_{0,0} \\ p_{1,0} \\
\vdots \\
p_{N,0} \\ \hdashline[0.5pt/1pt]
p_{0,1} \\ 
p_{1,1} \\ 
\vdots \\
p_{N,1} \\ \hdashline[0.5pt/1pt]
p_{0,2} \\
p_{1,2} \\
\vdots  \\
p_{N,2}
\end{bmatrix}
\end{equation}
Then along $R_{1}$, the matrix $PBP^{T}$ is demonstrated in~\cref{fig:smooth_decomp}.
\begin{figure}[!h]
\centering
\includegraphics[width=0.7\linewidth]{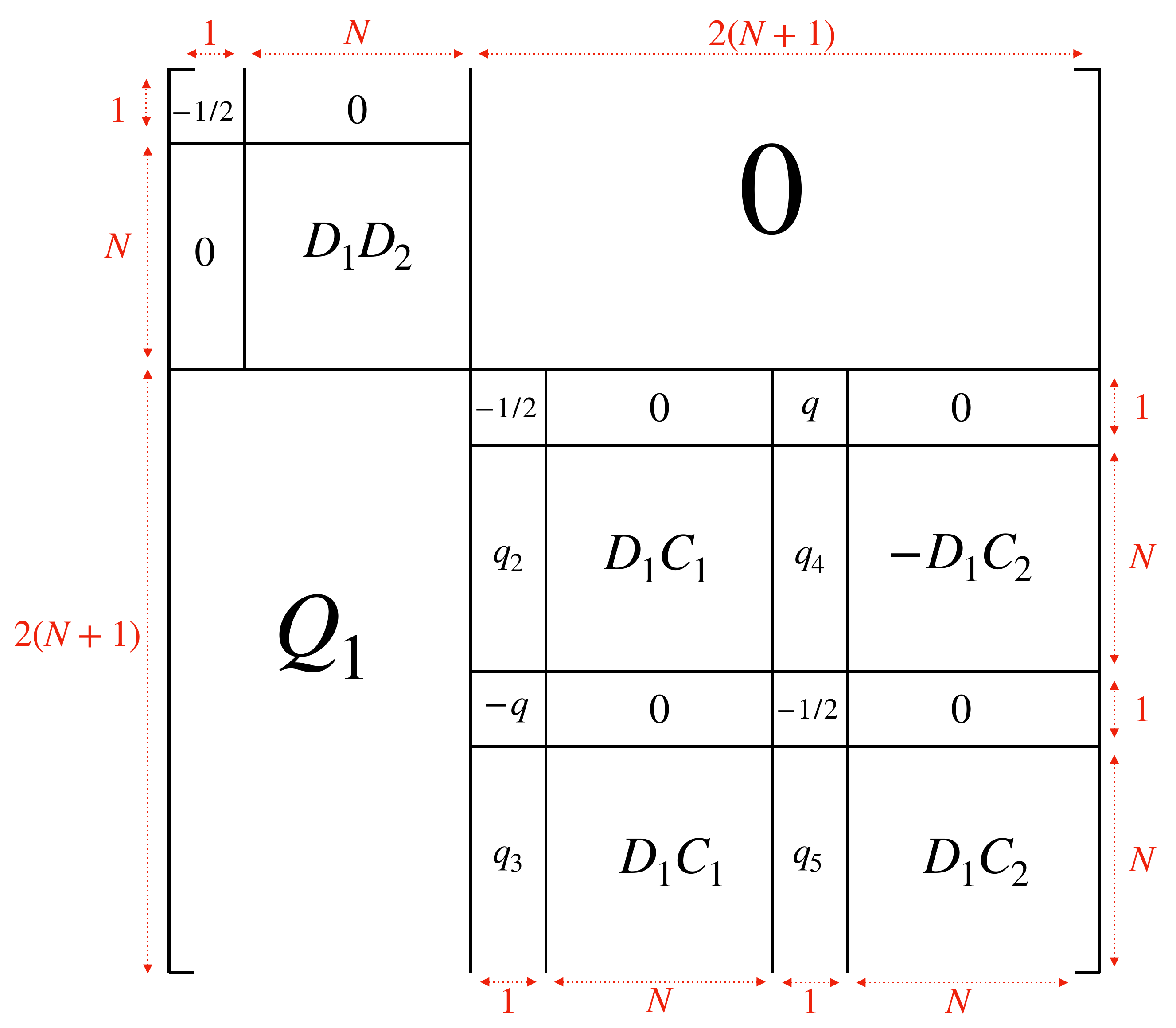}
\caption{Structure of the matrix $PBP^{T}$}
\label{fig:smooth_decomp}
\end{figure}
The matrices $D_{1},D_{2}$ are diagonal and are given by
\begin{equation}
D_{1} = \begin{bmatrix}
\sin{(\theta_{3})} & & & & \\
& \sin{(2 \theta_{3})} & & &  \\
& & \ddots & & \\
& & & \sin{((N-1) \theta_{3})} & \\
& & & & \sin{(N \theta_{3})}
\end{bmatrix} \, ,
\quad
D_{2} = -\frac{1}{2}\begin{bmatrix}
\eta_{1} & & & & \\
& \eta_{2} & & & \\
& & \ddots & & \\
& & & \eta_{N-1} & \\
& & & & \eta_{N}
\end{bmatrix} \, .
\end{equation}
The matrices $C_{1},C_{2}$ are Cauchy matrices whose entries are given by
\begin{equation}
C_{1,i,j} = \frac{1}{\beta_{i,1}-j} \, , C_{2,i,j} = \frac{1}{\beta_{i,2}-j} \, .
\end{equation}
Since we have assumed $\theta_{1}/\pi,\theta_{2}/\pi,\theta_{3}/\pi$, to be
irrational, we note that $\eta_{i}>0$ and that
$\sin{(m\theta_{3})} \neq 0$ for all $m\neq 0$. 
Thus, the diagonal matrices $D_{1},D_{2}$ are invertible.
Furthermore on $(a,0)$, neither of $\beta_{i,1}$ or $\beta_{i,2}$,
take on integer values~\cref{lem_transm}. 
Thus, the Cauchy matrices $C_{1},C_{2}$ are invertible.

Let $T$ denote the bottom-right $2(N+1)\times 2(N+1)$ block.
Then from the structure of $P\bB P^{T}$ and the fact that
the diagonal matrix $D_{1}D_{2}$ is invertible, it is clear that $B$ is invertible
if and only if $T$ is invertible. 
\begin{remark}
The matrix $T$ is the mapping from the coefficients of the singular
basis of solutions for the transmission problem with angle $\pi \theta_{3}$
and material parameter $a$ to the corresponding coefficients of
the Taylor expansion of the potential on the edges $(2,3),(3,1)$.
The invertibility of $T$ follows from the analysis 
in~\cite{hoskins2018numerical}.
We present the proof here in terms of the notation used in this paper.
\end{remark}

Upon applying an appropriate permutation matrix $P_{2}$
to $T$ from the right and the left, we note that 
\begin{equation}
P_{2} T P_{2}^{T} = 
\begin{bmatrix}
-1/2 & q & 0 & 0 \\
-q & -1/2 & 0 & 0 \\
q_{2} & q_{4} & D_{1}C_{1} & -D_{1}C_{2} \\
q_{3} & q_{5} & D_{1}C_{1} & D_{1} C_{2}
\end{bmatrix} \, .
\end{equation}
The matrix $P_{2}TP_{2}^{T}$ is invertible if 
and only if it's bottom right $2N \times 2N$ 
corner is invertible. 
Let $I_{N}$ denote the $N\times N$ identity matrix, 
then the bottom right corner of $P_{2}TP_{2}^{T}$ factorizes as
as 
\begin{equation}
\begin{bmatrix}
D_{1} & 0 \\
0 & D_{1} 
\end{bmatrix}
\begin{bmatrix}
I_{N} & -I_{N} \\
I_{N} & I_{N} 
\end{bmatrix}
\begin{bmatrix}
C_{1} & 0 \\ 
0 & C_{2}
\end{bmatrix} \, ,
\end{equation}
which is clearly invertible since the matrices $D_{1},C_{1},C_{2}$ 
are invertible.

Finally, using all of these results, it follows that the matrix $\bB$ 
is invertible for all $(a,0) = R_{1}$.
Since all of the quantities involved are analytic, on every compact subset
of $S^{1}_{N}$, we conclude that the matrix $\bB$ is invertible in
an open neighborhood $R_{1}\subset \tilde{S}^{1}_{N} \subset S^{1}_{N}$.
By construction $|\tilde{S}^{1}_{N}|>0$.


\section{Analysis of $\Kneu$ \label{sec:appendix-neu}}
All the proofs for the analysis of $\Kneu$ are similar to
the corresponding proofs of $\Kdir$.
We only present the analogs of~\cref{lem:dledge,lem:pot_fwd}.

In the following lemma we present the directional derivative 
of a single layer potential defined on straight line segment
with density $s^{\beta}$ at an arbitrary point near the boundary.
Here $s$ is the distance along the segment,
at an arbitrary point 
near the boundary. 
\begin{lemma}
\label{lem:dastedge}
Suppose that $\Gamma$ is an edge of unit length oriented
along an angle $\pi\theta$, parameterized by 
$s(\cost,\sint)$, $0<s<1$. 
Suppose that $\bx= t(\cos{(\theta+\theta_{0})}, \sin{(\theta+\theta_{0})}$
and $\bn = (-\sin{(\theta+\theta_{0})},\cos{(\theta+\theta_{0})})$
(see~\cref{fig_illus_ints})
where $0<t<1$, and $\bx \not \in \Gamma$.
Suppose that $\sigma(s) = s^{\beta-1}$ for $0<s<1$, where
$\beta \geq 1/2$.
If $\beta$ is not an integer, then
\begin{equation}
\nabla \cS[\sigma](\bx) \cdot \bn = 
-\frac{\sin{(\beta(\pi-\theta_{0}))}}{2\sinb}t^{\beta-1}
-\frac{1}{2\pi}\sum_{k=1}^{\infty} \frac{\sin{(k\theta_{0})}}{\beta-k} t^{k-1} \, .
\end{equation}
If $\beta=m$ is an integer, then
\begin{equation}
\nabla \cS_{\Gamma}[\sigma](\bx) \cdot \bn = -\frac{(\pi-\theta_{0})\cos{(m \theta_{0})}}{2\pi}t^{m-1}
+\frac{\sin{(m\theta_{0})}}{2\pi} t^{m-1} \log{(t)}
-\frac{1}{2\pi}\sum_{\substack{k=1\\k \neq m}}^{\infty} 
\frac{\spt{k\theta_{0}}}{m-k} t^{k-1} \, .
\end{equation}
\end{lemma}

In the following lemma, 
we compute the potential
at a triple junction with angles $\pi \theta_{1},\pi \theta_{2}, 
\pi \theta_{3}$, and material parameters $\bd = 
(d_{(1,2)},d_{(2,3)},d_{(3,1)})$
(see~\cref{fig:trip_geom}).

\begin{lemma}
\label{lem:pot_fwd_neu}
Consider the geometry setup of the single vertex problem
presented in~\cref{sec:mainres}.
For a constant vector $\bv \in \mathbb{R}^{3}$, suppose 
that the density on the edges is of the form
\begin{equation}
\sigma =
\begin{bmatrix}\sigma_{1,2} \\ \sigma_{2,3} \\ \sigma_{3,1}
\end{bmatrix} = 
\bw t^{\beta-1}
\end{equation}
If $\beta$ is not an integer, then 
\begin{equation}
\Kdir[\sigma] = 
-\frac{1}{2\sinb} \Aneu(d_{3,1},d_{1,2},\beta) \bw
t^{\beta}
- \sum_{k=1}^{\infty} \frac{1}{\beta-k}\bC(\bd,k) \bw 
t^{k-1} \, ,
\end{equation}
where
$\Aneu$ is defined in~\cref{eq:defAneu}, and $\bC(\bd,k)$
is defined in~\cref{eq:defcsmooth}.
If $\beta = m$ is an integer, then
\begin{equation}
\begin{aligned}
\Kneu[\sigma] &= 
-\frac{(-1)^{m}}{2\pi}\Aneu(d_{3,1},d_{1,2},m) \bw
t^{m} \log{(t)}
- \sum_{\substack{k=1\\k\neq m}}^{\infty} \frac{1}{m-k}\bC(\bd,k)\bw t^{k-1} - 
\bCdiag(\bd,m) \bw t^{m-1}\, ,
\end{aligned}
\end{equation}
where $\bCdiag$ is defined in~\cref{eq:defcsmoothint}.
\end{lemma}
\begin{proof}
The result follows from repeated application of~\cref{lem:dastedge}
for computing $\cD^{\ast}_{(l,m):(i,j)} \sigma_{(i,j)}$.
\end{proof}
The proof of~\cref{thm_fwd_neu} then follows immediately 
from~\cref{lem:pot_fwd_neu}.

In the following lemma, we prove that $\beta_{i,j},S_{i,j}$, 
$i=1,2,\ldots $, $j=0,1,2$, defined in~\cref{subsec:beta_v_existence}
satisfy $\beta_{i,j}(a,b)>1/2$ for all $(a,b)$ in an open subset 
$T_{i,j}\subset S_{i,j}$.

\begin{lemma}
\label{lem:betahalf}
Suppose that $\beta_{i,j},S_{i,j}$, $i=1,2,\ldots$, $j=0,1,2$,
are as defined in~\cref{subsec:beta_v_existence}.
Then there exists an open subset $T_{i,j}\subset S_{i,j}$,
such that $\beta_{i,j}(a,b)>1/2$ for all $(a,b) \in T_{i,j}$. 
Moreover for any $N>0$, $\cap_{i=1}^{N+1}\cap_{j=0}^{2} |T_{i,j}|>0$.
\end{lemma}
\begin{proof}
Since $\beta_{i,0} = i$, the statement is trivially true 
with $T_{i,j} = (-1,1)^2$.
Since $\beta_{i,j} = z_{i}^{\pm}(\delta,\theta)$ 
on $a=0,b=0$, or $c=0$, for appropriate parameters
$\delta,\theta$, we conclude that $\beta_{i,j}>1/2$,
on $a=0,b=0$, or $c=0$, for $i=1,2,\ldots$, $j=1,2$.
Since $\beta_{i,j}$ are analytic on $S_{i,j}$, 
there exists an open subset containing the segments
$a=0,b=0$, or $c=0$, which we denote by $T_{i,j}$,
such that $\beta_{i,j}(a,b)>1/2$ for all
$(a,b) \in T_{i,j}$.
Since each $T_{i,j}$ is an open subset of $(-1,1)^2$, 
containing $\cup_{k=1}^{6} R_{k}$, we conclude that
$|\cap_{i=1}^{N+1}\cap_{j=0}^{2} T_{i,j}|>0$.

\end{proof}





\end{document}